\documentclass[a4paper,reqno]{amsart}
\pdfoutput=1
\usepackage[utf8]{inputenc}
\usepackage[T1]{fontenc}
\usepackage{ae}
\usepackage{aecompl}
\usepackage{lmodern}
\title{Calabi--Yau properties of Postnikov diagrams}
\author{Matthew Pressland}
\address{Matthew Pressland\\School of Mathematics \& Statistics\\University of Glasgow\\University Place\\Glasgow\\G20 8LR\\United Kingdom}
\email{\href{mailto:Matthew.Pressland@glasgow.ac.uk}{Matthew.Pressland@glasgow.ac.uk}}
\urladdr{\url{http://mdpressland.github.io}}
\date{July 15, 2022}
\usepackage[hmargin=3cm,vmargin=3.5cm]{geometry}
\usepackage{amsmath}
\usepackage{amssymb}
\usepackage{amsfonts}
\usepackage{amsthm}
\usepackage{mathrsfs}
\usepackage{booktabs}
\usepackage{bm}
\usepackage[usenames,dvipsnames,svgnames]{xcolor}
\usepackage[colorlinks, urlcolor=Navy, linkcolor=Navy, citecolor=Navy]{hyperref}
\usepackage{microtype}
\usepackage{enumitem}
\makeatletter
\newcommand{\myitem}[1][]{%
\item[#1]\protected@edef\@currentlabel{#1}\ignorespaces%
}
\makeatother

\usepackage{caption}
\usepackage[backend=biber,citestyle=numeric-comp,bibstyle=numeric,maxbibnames=99,giveninits=true,doi=false,isbn=false,url=false,eprint=true]{biblatex}

\renewbibmacro{in:}{%
  \ifentrytype{article}{}{\printtext{\bibstring{in}\intitlepunct}}}

\DeclareFieldFormat[article,inbook,incollection,inproceedings,patent]{title}{#1}
\DeclareFieldFormat[thesis,unpublished]{title}{{\it #1}}

\DeclareFieldFormat[online]{date}{Preprint (#1)}

\DeclareFieldFormat[article]{volume}{\mkbibbold{#1}}
\renewbibmacro*{volume+number+eid}{%
  \printfield{volume}%
  \setunit*{\addnbspace}
  \printfield{number}%
  \setunit{\addcomma\space}%
  \printfield{eid}}
\DeclareFieldFormat[article]{number}{}

\DeclareFieldFormat[book,incollection]{number}{Vol.~#1}
\renewbibmacro*{series+number}{%
  \iffieldundef{series}{}
    {\printfield{series}%
      \iffieldundef{number}{}{\setunit*{\addcomma\addspace}%
       \printfield{number}%
       \newunit}}}

\renewbibmacro*{publisher+location+date}{%
  \printlist{publisher}%
  \setunit*{\addcomma\space}%
  \usebibmacro{date}%
  \newunit}
  
\DeclareFieldFormat{eprint:arxiv}{%
\ifentrytype{online}
  {\ifhyperref
    {\href{http://arxiv.org/abs/#1}{\nolinkurl{arXiv:#1}}}
    {\nolinkurl{arXiv:#1}}
   \iffieldundef{eprintclass}
    {}
    {{\tt \mkbibbrackets{\thefield{eprintclass}}}}}
  {\iffieldundef{eprintclass}
    {\mkbibparens{\ifhyperref
    {\href{http://arxiv.org/abs/#1}{\nolinkurl{arXiv:#1}}}
    {\nolinkurl{arXiv:#1}}}}
    {\mkbibparens{\ifhyperref
    {\href{http://arxiv.org/abs/#1}{\nolinkurl{arXiv:#1}}}
    {\nolinkurl{arXiv:#1}}
    {\tt \mkbibbrackets{\thefield{eprintclass}}}}}}}
\usepackage{color}
\usepackage{tikz}
  \usetikzlibrary{calc}
 \usetikzlibrary{arrows,decorations.markings}
\usepackage{tikz-cd}
\pgfarrowsdeclarecombine{twohead}{twohead}
{angle 90}{angle 90}{angle 90}{angle 90}

\newcommand{\strandcolor}{black}
\newcommand{\graphcolor}{black}
\newcommand{\quivcolor}{black}
\newcommand{\frozcolor}{black}
\newcommand{\bdrycolor}{gray}

\tikzset{strand/.style={\strandcolor,dashed},
 boundary/.style={ultra thick, \bdrycolor}}
\tikzset{bipedge/.style={\graphcolor, thick}}
\tikzset{quivarrow/.style={\quivcolor, -latex}}
\tikzset{frozarrow/.style={\frozcolor, -latex, very thick}}

\newcommand{\strarrow}{\arrow{angle 90}}
\newcommand{\bstart}{130} 
\newcommand{\fifth}{72} 
\newcommand{\dotrad}{1.4pt} 

\usepackage[capitalise,compress]{cleveref}
\Crefname{eg}{Example}{Examples}
\Crefname{thm}{Theorem}{Theorems}
\Crefname{conj}{Conjecture}{Conjectures}
\Crefname{prop}{Proposition}{Propositions}
\Crefname{lem}{Lemma}{Lemmas}
\Crefname{defn}{Definition}{Definitions}
\Crefname{thm*}{Theorem}{Theorems}
\Crefname{conj*}{Conjecture}{Conjectures}
\Crefname{ques}{Question}{Questions}
\Crefname{rem}{Remark}{Remarks}

\numberwithin{equation}{section}

\usepackage{todonotes}

\theoremstyle{plain}
\newtheorem{thm}{Theorem}[section]
\newtheorem{thm*}{Theorem}
\newtheorem{lem}[thm]{Lemma}
\newtheorem{cor}[thm]{Corollary}
\newtheorem{prop}[thm]{Proposition}

\theoremstyle{definition}
\newtheorem{defn}[thm]{Definition}
\newtheorem{eg}[thm]{Example}

\newtheorem{rem}[thm]{Remark}
\theoremstyle{remark}

\renewcommand{\subset}{\subseteq}

\newcommand{\K}{\mathrm{K}}

\newcommand{\NN}{\mathbb{N}}
\newcommand{\CC}{\mathbb{C}}
\newcommand{\KK}{\mathbb{K}}
\newcommand{\ZZ}{\mathbb{Z}}

\newcommand{\frobcat}{\mathcal{E}}
\newcommand{\abcat}{\mathcal{A}}
\newcommand{\ctcat}{\mathcal{T}}

\DeclareMathOperator{\CM}{CM}

\DeclareMathOperator{\fgmod}{mod}
\DeclareMathOperator{\fd}{fd}

\DeclareMathOperator{\Ext}{Ext}
\DeclareMathOperator{\Hom}{Hom}
\DeclareMathOperator{\stabEnd}{\underline{End}}

\DeclareMathOperator{\End}{End}
\DeclareMathOperator{\add}{add}
\DeclareMathOperator{\GP}{GP}
\DeclareMathOperator{\MCM}{MCM}

\DeclareMathOperator{\gldim}{gldim}

\DeclareMathOperator{\im}{im}
\DeclareMathOperator{\stabGP}{\underline{GP}}

\DeclareMathOperator{\per}{per}
\DeclareMathOperator{\catrad}{Rad}

\DeclareMathOperator{\indec}{indec}
\DeclareMathOperator{\Ob}{Ob}

\newcommand{\cpa}[2]{#1\langle\hspace{-0.1em}\langle #2\rangle\hspace{-0.1em}\rangle}

\newcommand{\op}{\mathrm{op}}

\newcommand{\iso}{\cong}
\newcommand{\isoto}{\stackrel{\sim}{\to}}
\newcommand{\tens}{\mathbin{\otimes}}

\newcommand{\bigdsum}{\bigoplus}
\newcommand{\union}{\cup}

\newcommand{\map}[1]{\stackrel{#1}{\longrightarrow}}

\newcommand{\powser}[2]{#1[\hspace{-0.1em}[#2]\hspace{-0.1em}]}

\newcommand{\Endalg}[2]{\End_{#1}(#2)^{\op}}
\newcommand{\stabEndalg}[2]{\stabEnd_{#1}(#2)^{\op}}

\newcommand{\head}[1]{h#1}
\newcommand{\tail}[1]{t#1}
\newcommand{\idemp}[1]{e_{#1}}

\newcommand{\clustcat}[1]{\mathcal{C}_{#1}}
\newcommand{\clustalg}[1]{\mathscr{A}_{#1}}

\newcommand{\clust}{\mathcal{C}}
\newcommand{\subsets}[2]{\binom{[#2]}{#1}}
\newcommand{\Gra}[2]{\mathrm{Gr}_{#1}^{#2}}
\newcommand{\coneGra}[2]{\widehat{\mathrm{Gr}}{\vphantom{\mathrm{Gr}}}_{#1}^{#2}}
\newcommand{\posvar}{\Pi}
\newcommand{\openposvar}{\Pi^\circ}
\newcommand{\coneposvar}{\widehat{\Pi}}
\newcommand{\coneopenposvar}{\widehat{\Pi}^\circ}
\newcommand{\posit}{\mathcal{P}}
\newcommand{\neck}{\mathcal{I}}
\newcommand{\Plueck}[1]{\Delta_{#1}}

\newcommand{\simp}[1]{S_{#1}}
\newcommand{\res}[1]{\mathbf{P}(#1)}

\newcommand{\frjac}[3]{\mathcal{J}(#1,#2,#3)}
\newcommand{\cyc}{\mathrm{cyc}}

\newcommand{\rad}[2][]{\operatorname{rad}^{#1}(#2)}

\newcommand{\Kdual}{\mathrm{D}}

\newcommand{\stab}[1]{\underline{#1}}

\newcommand{\set}[1]{\left\{#1\right\}}
\newcommand{\Span}[1]{\langle#1\rangle}

\newcommand{\der}[1]{\partial_{#1}}
\newcommand{\leftder}[1]{\partial^l_{#1}}
\newcommand{\rightder}[1]{\partial^r_{#1}}

\newcommand{\rel}[1]{\rho_{#1}}
\newcommand{\close}[1]{\overline{#1}}
\newcommand{\wt}{\mathrm{wt}}

\newcommand{\mut}{\mathrm{m}}

\newcommand{\minpath}[2]{p_{#2#1}}

\newcommand{\Head}[1]{\mathrm{H}_{#1}}
\newcommand{\Tail}[1]{\mathrm{T}_{#1}}
\newcommand{\mHead}[1]{\mathrm{H}_{#1}^{\mut}}
\newcommand{\mTail}[1]{\mathrm{T}_{#1}^{\mut}}

\newcommand{\QGr}[2]{\mathrm{Gr}_{#1}(#2)}
\newcommand{\Eul}[2]{\langle#1,#2\rangle}
\newcommand{\FKcc}[1]{\varphi^{#1}}

\begin{document}

\begin{abstract}
We show that the dimer algebra of a connected Postnikov diagram in the disc is bimodule internally $3$-Calabi--Yau in the sense of the author's earlier work \cite{presslandinternally}. As a consequence, we obtain an additive categorification of the cluster algebra associated to the diagram, which (after inverting frozen variables) is isomorphic to the homogeneous coordinate ring of a positroid variety in the Grassmannian by a recent result of Galashin and Lam \cite{galashinpositroid}. We show that our categorification can be realised as a full extension closed subcategory of Jensen--King--Su's Grassmannian cluster category \cite{jensencategorification}, in a way compatible with their bijection between rank $1$ modules and Plücker coordinates.
\end{abstract}
\maketitle

\section{Introduction}

\subsection{Main results}

A Postnikov diagram, or alternating strand diagram \cite{postnikovtotal}, is a combinatorial object consisting of a collection of strands in the disc. It encodes a great deal of further combinatorial, algebraic and geometric data, mostly connected to questions regarding total positivity in the Grassmannian. From the point of view of this paper, the key piece of data encoded by a Postnikov diagram $D$ is that of an ice quiver with potential $(Q_D,F_D,W_D)$. Here $Q_D$ is a quiver with a distinguished subquiver $F_D$ which we call frozen, so that the pair $(Q_D,F_D)$ is an ice quiver, and the potential $W_D$ is, loosely speaking, a linear combination of cycles in $Q_D$ (see Definition~\ref{d:iqp} for a precise definition). The quiver $Q_D$ and the vertex set of $F_D$ determine a cluster algebra $\clustalg{D}$ with frozen variables; the reader unfamiliar with this construction can find details in, for example, Keller's survey \cite{kellercluster}. In this paper, we will explain how to use the triple $(Q_D,F_D,W_D)$ to construct an additive categorification of this cluster algebra, in the case that $D$ is connected.

The construction proceeds via the frozen Jacobian algebra $A_D$ associated to $(Q_D,F_D,W_D)$, also known as the dimer algebra of $D$. This algebra is a quotient of the complete path algebra of $Q_D$ by the closure of an ideal generated by relations computed using the data of the arrow set of $F_D$ (or, more directly, the complement of this set) and the potential $W_D$; again, precise definitions are found in Definition~\ref{d:iqp}. This algebra has a distinguished idempotent $e$, given by the sum of vertex idempotents over the vertices of $F_D$, from which we obtain a \emph{boundary algebra} $B_D=eA_De$. We will show that the category $\GP(B_D)$ of Gorenstein projective $B_D$-modules is our desired categorification.

\begin{thm*}[cf.~Theorem~\ref{t:icytofrobcat} and Theorem~\ref{t:2cy-realisation}]
\label{t:mainthm}
Let $D$ be a connected Postnikov diagram, and let $B_D=eA_De$ be the boundary algebra of its dimer algebra. Then the category $\GP(B_D)$ is an additive categorification of the cluster algebra $\clustalg{D}$.
\end{thm*}

While we use the term `additive categorification' for brevity in the introduction, we acknowledge that it does not have a widely accepted formal definition in this context, and refer the reader to the more precise statements below, which we hope justify our use of this terminology.

The main step in proving Theorem~\ref{t:mainthm} is to establish a certain Calabi--Yau property for the algebra $A_D$, relative to the boundary idempotent $e$, as referred to in the title. This result is likely to be of independent interest, and is analogous to Broomhead's theorem \cite[Thm.~7.1]{broomheaddimer} concerning Calabi--Yau properties of algebraically consistent dimer models on the torus (see also Davison \cite{davisonconsistency} for higher genus surfaces).

\begin{thm*}[cf.\ Theorem~\ref{t:bi3cy}]
\label{t:cy-intro}
Let $D$ be a connected Postnikov diagram, with dimer algebra $A_D$. Then $A_D$ is bimodule internally $3$-Calabi--Yau with respect to the idempotent $e$ given by summing the vertex idempotents for the vertices of $F_D$.
\end{thm*}

The general definition of a bimodule internally $3$-Calabi--Yau algebra, and the role of this property in constructing categorifications, is explained in \cite{presslandinternally}. We will not give this general definition here; since $A_D$ is presented as a frozen Jacobian algebra, we can work instead with a sufficient condition implying the Calabi--Yau property, which appears in Theorem~\ref{t:bi3cy}. 

\subsection{Positroid varieties}
\label{ss:positroids}

Postnikov diagrams and their associated cluster algebras are important in studying positroid varieties in the Grassmannian---while this connection has only a minor logical role in the paper, we recall it here to help contextualise and motivate our results. Positroid varieties were introduced by Knutson, Lam and Speyer \cite{knutsonpositroid}, and are closely related to various stratifications of the Grassmannian \cite{lusztigtotal,rietschclosure,brownpoisson,goodearlpoisson}; see \cite[\S5.3]{knutsonpositroid} for a discussion of this relationship. Intersecting a positroid variety with the totally non-negative Grassmannian produces a cell in Postnikov's decomposition of this space, the study of which is one of the main motivations of the work in \cite{postnikovtotal}.

In this context, a positroid is a subset $\posit$ of the set $\subsets{k}{n}$ of $k$-element subsets of $\set{1,\dotsc,n}$ (see \cite[\S1.4]{mullertwist} for the defining properties). Each positroid is uniquely determined by a \emph{necklace}, a cyclically ordered set of $n$ elements of $\subsets{k}{n}$ (see \cite[Defn.~4.2]{ohweak}). Strictly speaking, it is more compatible with our conventions to use the reverse necklace (see \cite[\S2.1]{mullertwist} and \cite[\S8]{canakciperfect}), but we will abuse terminology by referring to the set $\neck\subset\posit$ of elements of this reverse necklace as simply `the necklace $\neck$ of $\posit$'. Recall that the Grassmannian $\Gra{k}{n}$ of $k$-dimensional subspaces of $\CC^n$ has as a standard generating set of homogeneous coordinates the Plücker coordinates $\Plueck{I}$ for $I\in\subsets{k}{n}$.

\begin{defn}
\label{d:positroid-var}
Given a positroid $\posit\subset\subsets{k}{n}$, the (closed) \emph{positroid variety} $\posvar(\posit)$ is the subvariety of the Grassmannian $\Gra{k}{n}$ on which $\Plueck{I}=0$ for $I\notin\posit$. The \emph{open positroid variety} $\openposvar(\posit)$ is the subvariety of $\posvar(\posit)$ on which additionally $\Plueck{I}\ne0$ for $I\in\neck$. We denote the cones on these varieties, which are defined by the same conditions inside the affine cone $\coneGra{k}{n}$ on the Grassmannian, by $\coneposvar(\posit)$ and $\coneopenposvar(\posit)$ respectively.
\end{defn}

A Postnikov diagram $D$ determines a positroid $\posit_D$ with necklace $\neck_D$, and a maximal weakly separated (or maximal non-crossing) collection \cite[Defn.~4.3]{ohweak} $\clust_D$ with $\neck_D\subset\clust_D\subset\posit_D$. In particular, $D$ determines the numbers $k$ and $n$ such that $\posit_D$, $\neck_D$ and $\clust_D$ consist of $k$-element subsets of $\set{1,\dotsc,n}$; see Definition~\ref{d:postnikov-diag}. The elements of $\clust_D$ are attached to the alternating regions of $D$, or equivalently to the vertices of the associated ice quiver $Q_D$; the precise construction is given in Remark~\ref{r:positroids}. While many different Postnikov diagrams can determine the same positroid and necklace, it is rarer for two diagrams to determine the same weakly separated collection---indeed, given a positroid $\posit$ with necklace $\neck$, Oh--Postnikov--Speyer \cite[Thm.~1.5]{ohweak} show that the maximal weakly separated collections $\clust$ with $\neck\subset\clust\subset\posit$ are in bijection with Postnikov diagrams $D$ such that $\posit_D=\posit$ (equivalently, $\neck_D=\neck$), up to isotopy and certain moves (see Figure~\ref{f:untwist}).

We may interpret the collection $\clust_D$ as a set of functions on $\coneopenposvar(\posit_D)$ by restricting the Plücker coordinates $\Delta_I$ for $I\in\clust_D$ to this subvariety of $\coneGra{k}{n}$. Since these functions are attached to the vertices of $Q_D$, we thus obtain a map $\clustalg{D}\to\CC(\coneopenposvar(\posit_D))$, from the abstract cluster algebra $\clustalg{D}$ associated to $D$ to the ring of rational functions on $\coneopenposvar(\posit_D)$, sending each initial cluster variable to the restricted Plücker coordinate corresponding to the same quiver vertex. A long-standing conjecture, formalised by Muller and Speyer \cite[Conj.~3.4]{mullercluster} and recently verified by Galashin and Lam \cite{galashinpositroid} (see also \cite{serhiyenkoclusterstructures} for the special case of Schubert varieties), is the following.

\begin{thm}[\cite{galashinpositroid}]
\label{t:GL}
The specialisation map $\clustalg{D}\to\CC(\coneopenposvar(\posit_D))$ above is injective and its image is precisely the coordinate ring $\CC[\coneopenposvar(\posit_D)]$.
\end{thm}

Thus the coordinate ring $\CC[\coneopenposvar(\posit_D)]$ has a cluster algebra structure, with initial seed $(Q_D,\clust_D)$, isomorphic to the cluster algebra we categorify in Theorem~\ref{t:mainthm}. While $\openposvar(\posit_D)$ is an affine variety, and so it is slightly unusual to consider its affine cone, it is this cone on which restricted Plücker coordinates are functions. By contrast, the coordinate ring $\CC[\openposvar(\posit_D)]$ is the degree zero part of $\CC[\coneopenposvar(\posit_D)]$ in the $\ZZ$-grading in which Plücker coordinates have degree $1$. Said differently, the grading determines a $\CC^\times$-action on $\CC[\coneopenposvar(\posit_D)]$ by $\lambda\cdot f=\lambda^df$ when $f$ has degree $d$, and $\CC[\openposvar(\posit_D)]$ is the invariant subring $\CC[\coneopenposvar(\posit_D)]^{\CC^\times}$ for this action.

\begin{rem}
For Theorem~\ref{t:GL} it is necessary to adopt the convention that the frozen variables of $\clustalg{D}$ are invertible, since the images of these frozen variables in $\CC[\coneopenposvar(\posit_D)]$ are precisely the Plücker coordinates $\Plueck{I}$ for $I\in\neck_D$, which are non-zero on $\coneopenposvar(\posit_D)$ by definition. This convention is not relevant to Theorem~\ref{t:mainthm} however, and so we will freely compare our results with those of authors who do not invert frozen variables, such as \cite{jensencategorification}.
\end{rem}

As already observed, the positroid $\posit_D$ and its corresponding variety $\openposvar(\posit_D)$ do not determine $D$ uniquely; choosing a different Postnikov diagram with positroid $\posit_D$ amounts to choosing a different initial seed in the cluster algebra structure on $\CC[\coneopenposvar(\posit_D)]$. This is reflected in our categorical picture by the fact (Corollary~\ref{c:dec-perm}) that if $\posit_D=\posit_{D'}$ for two connected Postnikov diagrams $D$ and $D'$, then the two algebras $B_D$ and $B_D'$ are isomorphic, and hence the categories $\GP(B_D)$ and $\GP(B_{D'})$ from Theorem~\ref{t:mainthm} are equivalent.

The best understood case of these constructions is that in which $\posit_D=\subsets{k}{n}$ is the \emph{uniform positroid}, in which the notion of a maximal weakly separated collection from \cite{ohweak} coincides with that of Leclerc--Zelevinsky \cite{leclercquasicommuting}. In this case we have $\posvar(\posit_D)=\Gra{k}{n}$, the open subvariety $\openposvar(\posit_D)$ is determined by the non-vanishing of those Plücker coordinates labelled by cyclic intervals, and Muller--Speyer's conjecture was verified much earlier by Scott \cite{scottgrassmannians}. (The case $k=2$ even appears in Fomin--Zelevinsky's first paper on cluster algebras \cite{fomincluster1}.) Scott in fact proves more, namely that a completely analogous specialisation map to that of Theorem~\ref{t:GL} is an isomorphism from the cluster algebra with non-invertible frozen variables determined by $Q_D$ to the homogeneous coordinate ring of $\Gra{k}{n}$. However, this depends on special properties of the uniform positroid, and there is no corresponding result for arbitrary Postnikov diagrams.

\subsection{The Grassmannian cluster category}
When $\posit_D$ is the uniform positroid, a categorification of $\clustalg{D}$ is provided by Jensen--King--Su \cite{jensencategorification} (see also Demonet--Luo \cite{demoneticequivers1} for $k=2$). This \emph{Grassmannian cluster category}, denoted $\CM(C)$, consists of Cohen--Macaulay modules over a certain algebra $C$ (depending on $k$ and $n$), defined directly via a quiver with relations. Some of the results in \cite{jensencategorification} rely on the categorification by Geiß--Leclerc--Schröer \cite{geisspartial} of a cluster algebra obtained from $\clustalg{D}$ by specialising a single frozen variable to $1$, or equivalently by removing a frozen vertex and all adjacent arrows from the quiver.

Jensen--King--Su's categorification is connected back to the combinatorics of Postnikov diagrams by work of Baur--King--Marsh \cite{baurdimer}, which is our main inspiration. They show \cite[Cor.~10.4]{baurdimer} that for any Postnikov diagram $D$ for which $\posit_D$ is the uniform positroid, there is an isomorphism $B_D\iso C$ between the boundary algebra of $D$, with which we describe our categorification, and the algebra $C$ describing the Grassmannian cluster category. Since $C$ has the property that $\CM(C)=\GP(C)$, our categorification coincides with Jensen--King--Su's in this case.

In our more general situation, because $B_D$ is defined as an idempotent subalgebra of the dimer algebra $A_D$, we do not typically know a presentation of $B_D$ via a quiver with relations. Our approach therefore differs significantly from that of Jensen--King--Su, since we must deduce properties of the category $\GP(B_D)$ from the combinatorics of the diagram $D$ and the representation theory of $A_D$, for which we do have such a presentation. This is done primarily via the internal Calabi--Yau property of Theorem~\ref{t:bi3cy} which, in the full strength established here, is also a new result in the case that $\posit_D$ is the uniform positroid.

For more general Postnikov diagrams $D$, a categorification analogous to that of Geiß--Leclerc--Schröer's in the uniform case is provided by Leclerc \cite{leclerccluster} in work on Richardson varieties; again he categorifies a cluster algebra obtained from $\clustalg{D}$ by specialising a frozen variable to $1$. Leclerc's category plays a crucial role in the construction of a cluster structure on $\CC[\coneopenposvar(\posit_D)]$ in \cite{serhiyenkoclusterstructures,galashinpositroid}. In this paper we will always consider the full quiver $Q_D$ and its associated cluster algebra---in particular, this means that in the case of the uniform positroid Theorem~\ref{t:mainthm} recovers Jensen--King--Su's category \cite{jensencategorification}, not Geiß--Leclerc--Schröer's \cite{geisspartial}.

Our category $\GP(B_D)$ is closely connected to the coordinate ring $\CC[\coneopenposvar(\posit_D)]$ (and its cluster algebra structure). To demonstrate this, we exhibit a cluster character on $\GP(B_D)$, in the sense of Fu--Keller \cite{fucluster}, providing a bijection between reachable rigid indecomposable objects of $\GP(B_D)$ and cluster variables of $\CC[\coneopenposvar(\posit_D)]$, and a bijection between reachable cluster-tilting objects of $\GP(B_D)$ and clusters of $\CC[\coneopenposvar(\posit_D)]$. We also recall results from joint work with Çanakçı and King \cite{canakciperfect} demonstrating even stronger combinatorial connections between the representation theory of $B_D$ and the geometry of $\coneopenposvar(\posit_D)$. Precisely, there is a fully faithful functor $\rho\colon\GP(B_D)\to\CM(C)$ from the categorification described here to the Grassmannian cluster category (for the appropriate Grassmannian). For each $I\in\subsets{k}{n}$, there is a combinatorially defined $C$-module $M_I$ corresponding to the Plücker coordinate $\Plueck{I}$, and we show (Theorem~\ref{t:rank-1-plueck}) that if $\rho(M)=M_I$, then the cluster character on $\GP(B_D)$ takes $M$ to the restriction of $\Plueck{I}$ to $\coneopenposvar(\posit_D)$.

\begin{rem}
Leclerc's categories \cite{leclerccluster} can be used to categorify a cluster algebra structure on $\CC[\openposvar(\posit_D)]$, i.e.\ on the coordinate ring of the open positroid variety itself rather than its affine cone, but in a way that depends on choosing a frozen variable $\Plueck{I}$ with $I\in\neck_D$. Indeed, the construction involves an isomorphism of $\openposvar(\posit_D)$ with the subvariety of $\coneopenposvar(\posit_D)$ on which $\Plueck{I}=1$; this is the graph of a global section of the line bundle $\coneopenposvar(\posit_D)\to\openposvar(\posit_D)$. In ring-theoretic terms, this isomorphism corresponds to splitting the inclusion
\[\CC[\openposvar(\posit_D)]=\CC[\coneopenposvar(\posit_D)]^{\CC^\times}\subset\CC[\coneopenposvar(\posit_D)]\]
using the map taking $f\in\CC[\coneopenposvar(\posit_D)]$ of degree $d$ to the $\CC^\times$-invariant function $f\Plueck{I}^{-d}$, recalling that $\Plueck{I}$ is invertible on $\openposvar(\posit_D)$.

By avoiding this arbitrary choice, our categorification better reflects the symmetries of $\openposvar(\posit_D)$. For example, in the case of the uniform positroid, both $\openposvar(\posit_D)$ and our category (which coincides with Jensen--King--Su's \cite{jensencategorification} in this case) carry an action of the order $n$ cyclic group coming from natural action of this group on the cyclically ordered set $[n]$, and by extension on the set $\subsets{k}{n}$ labelling Plücker coordinates. Since this action does not fix the frozen variables, but rather cyclically permutes them, choosing a frozen vertex breaks the symmetry, which indeed is no longer present in Geiß--Leclerc--Schröer's category \cite{geisspartial}.

In the case of the uniform positroid, the distinction between Jensen--King--Su's category and Geiß--Leclerc--Schröer's is even more critical, since in this case the homogeneous coordinate ring $\CC[\coneposvar(\posit_D)]=\CC[\coneGra{k}{n}]$ of the closed positroid variety carries its own cluster algebra structure, in which the $n$ frozen variables are non-invertible. For more general positroids $\posit$, the ring $\CC[\coneposvar(\posit)]$ is not a cluster algebra, but we still hope that the avoidance of symmetry-breaking in our constructions will allow us to use the algebra $B_D$ constructed here, and its category of Cohen--Macaulay modules, to study the closed varieties $\posvar(\posit)$ and $\coneposvar(\posit)$---see \cite[Prop.~8.6]{canakciperfect} for a first step in this direction.
\end{rem}


\medskip

The structure of the paper is as follows. In Section~\ref{s:postnikov-diags} we recall the definition of a Postnikov diagram and the construction of its dimer algebra. The necessary Calabi--Yau property of this algebra is established in Section~\ref{s:calabi-yau}, and used to construct the categorification $\GP(B_D)$ in Section~\ref{s:categorification}. In Section~\ref{s:no-loops} we show that $\GP(B_D)$ has a mutation class of cluster-tilting objects, again determined by $D$, whose quivers have no loops or $2$-cycles. This allows us to show in Section~\ref{s:cluster-character} that there is a cluster character on $\GP(B_D)$ inducing bijections between the cluster-tilting objects in this class and the clusters of $\clustalg{D}$, and between their indecomposable summands and the cluster (and frozen) variables of $\clustalg{D}$. We close in Section~\ref{s:Grassmannian} by explaining how our categorification is related to Jensen--King--Su's Grassmannian cluster category \cite{jensencategorification}, and how its cluster character relates to Galashin and Lam's isomorphism of $\clustalg{D}$ with the coordinate ring $\CC[\coneopenposvar(\posit_D)]$.

\section{Postnikov diagrams and their dimer algebras}
\label{s:postnikov-diags}

In this section, we introduce Postnikov diagrams in the disc, and their associated dimer algebras, which will be defined over a field $\KK$ of characteristic zero.

\begin{defn}[{\cite[\S14]{postnikovtotal}}, see also {\cite[\S3]{scottgrassmannians}} and {\cite[\S2]{baurdimer}}]
\label{d:postnikov-diag}
Let $\Sigma$ be an oriented disc, with $n\geq 3$ marked points on its boundary. A \emph{Postnikov diagram} $D$ in $\Sigma$ consists of $n$ oriented strands in $\Sigma$, starting and ending at the boundary marked points, subject to the following constraints.
\begin{itemize}
\item[(P0)]Each boundary marked point is the start point of exactly one strand, and the end point of exactly one strand.
\item[(P1)]The strands intersect in finitely many points, and each such crossing is transverse and involves only two strands.
\item[(P2)]Moving along a strand, the signs of its crossings with other strands alternate.
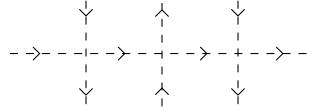
\[\begin{tikzpicture}
\draw[dashed]
(0,0) -- (4,0)
[postaction=decorate, decoration={markings,
 mark= at position 0.1 with \strarrow,
 mark= at position 0.38 with \strarrow,
 mark= at position 0.65 with \strarrow,
 mark= at position 0.9 with \strarrow}];
\draw[dashed]
(1,0.7) -- (1,-0.7)
[postaction=decorate, decoration={markings,
 mark= at position 0.15 with \strarrow,
 mark= at position 0.9 with \strarrow}];
\draw[dashed]
(2,-0.7) -- (2,0.7)
[postaction=decorate, decoration={markings,
 mark= at position 0.18 with \strarrow,
 mark= at position 0.92 with \strarrow}];
\draw[dashed]
(3,0.7) -- (3,-0.7)
[postaction=decorate, decoration={markings,
 mark= at position 0.15 with \strarrow,
 mark= at position 0.9 with \strarrow}];
\end{tikzpicture}\]
\item[(P3)]A strand may not intersect itself in the interior of the disc.
\item[(P4)]If two strands intersect twice, the strands are oriented in opposite directions between these intersection points. In other words, the configuration
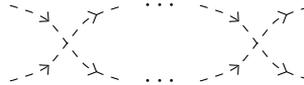
\[\begin{tikzpicture}
\node at (2,0.5) {$\cdots$};
\node at (2,-0.5) {$\cdots$};
\draw[dashed] plot[smooth]
coordinates {(0,-0.5) (0.5,-0.3) (1,0.3) (1.5,0.5)}
[postaction=decorate, decoration={markings,
 mark= at position 0.3 with \strarrow,
 mark= at position 0.8 with \strarrow}];
\draw[dashed] plot[smooth]
coordinates {(2.5,0.5) (3,0.3) (3.5,-0.3) (4,-0.5)}
[postaction=decorate, decoration={markings,
 mark= at position 0.3 with \strarrow,
 mark= at position 0.8 with \strarrow}];
\draw[dashed] plot[smooth]
coordinates {(0,0.5) (0.5,0.3) (1,-0.3) (1.5,-0.5)}
[postaction=decorate, decoration={markings,
 mark= at position 0.3 with \strarrow,
 mark= at position 0.8 with \strarrow}];
\draw[dashed] plot[smooth]
coordinates {(2.5,-0.5) (3,-0.3) (3.5,0.3) (4,0.5)}
[postaction=decorate, decoration={markings,
 mark= at position 0.3 with \strarrow,
 mark= at position 0.8 with \strarrow}];
\end{tikzpicture}\]
is not permitted, but the configuration obtained from this by reversing the orientation of one of the strands is legal.
\end{itemize}
For axioms (P2) and (P4) the intersection of two strands at the boundary, with its natural sign, is also treated as a crossing.

The diagram $D$ divides $\Sigma$ into \emph{regions}, the connected components of the complement of the strands, which have three types; the orientations of the strands around the boundary of such a region can be inconsistent (in which case we call the region \emph{alternating}), oriented anticlockwise (i.e.\ consistent with the orientation of $\Sigma$) or oriented clockwise. We will use the terminology \emph{anticlockwise} and \emph{clockwise} for regions with these boundary orientations. The alternating regions are further subdivided into \emph{boundary regions}, those incident with the boundary of $\Sigma$, and \emph{internal regions}.

We say that $D$ is \emph{connected} if the union of its strands is connected. This forces each boundary region to meet the boundary of $\Sigma$ in a single arc, and so there are exactly $n$ such regions.

There is an integer $k\in\{1,\dotsc,n\}$ such that each alternating region of $D$ is on the left of exactly $k$ strands. We say that the \emph{type} of $D$ is $(k,n)$, recalling that $n$ is the number of strands. One can also compute $k$ as the average number of boundary regions to the left of a strand. If $D$ has the property that every strand has exactly $k$ boundary regions on its left, then each strand must terminate at a marked point $k$ steps clockwise from its source, and in this case $D$ is sometimes \cite{baurdimer} called a \emph{$(k,n)$-diagram}. Diagrams of this kind are always connected.
\end{defn}

\begin{rem}
It is possible to consider Postnikov diagrams on $n<3$ strands, but they do not lead to interesting cluster algebras. For example, their associated ice quivers (see Definition~\ref{d:dimer-algebra} below) have only frozen vertices, and so describe a cluster algebra with only frozen variables, i.e.\ a polynomial ring. Thus the assumption that $n\geq 3$, which will be useful in Section~\ref{s:no-loops}, only excludes a small number of highly degenerate examples. Both Theorems~\ref{t:mainthm} and \ref{t:cy-intro} do hold when $n<3$, but they reduce to the claim that $B_D$ is an algebra of global dimension at most $3$ with $n$ indecomposable projective modules up to isomorphism, which follows immediately from observing that $B_D$ is the complete path algebra of an $n$-cycle in this case (or from the results of Section~\ref{s:calabi-yau}, which do not use the assumption $n\geq3$).
\end{rem}

A Postnikov diagram in the disc has an associated algebra, called its \emph{dimer algebra}, which, when $D$ is connected, is an instance of a frozen Jacobian algebra defined via an ice quiver with potential. We first give the necessary definitions for this general construction, and then explain how to read off the relevant data from a Postnikov diagram.

\begin{defn}[cf.\ {\cite[\S2]{presslandmutation}}]
\label{d:iqp}
An \emph{ice quiver} is a pair $(Q,F)$, where $Q$ is a quiver, and $F$ is a (not necessarily full) subquiver of $Q$. We call $F_0$, $F_1$ and $F$ the \emph{frozen} vertices, arrows and subquiver respectively. Vertices in $Q_0^\mut:=Q_0\setminus F_0$ will be called \emph{mutable}, whereas arrows in $Q_1^\mut:=Q_1\setminus F_1$ will be simply called \emph{unfrozen}. We assume for simplicity that $Q$ has no loops.

Let $\cpa{\KK}{Q}$ denote the complete path algebra of $Q$ (see \cite[Defn.~2.6]{presslandmutation}). This algebra is graded by path length, which makes its quotient $\cpa{\KK}{Q}_\cyc$ by the vector subspace spanned by commutators into a graded vector space. A \emph{potential} on $Q$ is an element $W\in\cpa{\KK}{Q}_\cyc$  expressible as a linear combination of homogeneous elements of degree at least $2$. We usually think of $W$ as a linear combination of cyclic paths in $Q$ (of length at least $2$), considered up to cyclic equivalence, i.e.\ the finest equivalence relation on such paths such that
\[a_n a_{n-1}\dotsm a_1\sim a_{n-1}\dotsm a_1a_n.\]
We call the triple $(Q,F,W)$ an \emph{ice quiver with potential}.

Let $p=a_n\dotsm a_1$ be a cyclic path, with each $a_i\in Q_1$, and let $a\in Q_1$ be any arrow. Then the \emph{cyclic derivative} of $p$ with respect to $a$ is
\[\der{a}{p}:=\sum_{a_i=a}a_{i-1}\dotsm a_1a_n\dotsm a_{i+1}.\]
Extending $\der{a}{}$ by linearity and continuity, we obtain a map $\cpa{\KK}{Q}_\cyc\to\cpa{\KK}{Q}$.
For an ice quiver with potential $(Q,F,W)$, we define the \emph{frozen Jacobian algebra}
\[\frjac{Q}{F}{W}=\cpa{\KK}{Q}/\close{\Span{\der{a}W:a\in Q_1\setminus F_1}}.\]
Here the closure is taken in the $J$-adic topology on $\cpa{\KK}{Q}$, where $J$ is the ideal generated by arrows.
\end{defn}

\begin{defn}
\label{d:dimer-algebra}
Let $D$ be a Postnikov diagram in the disc. We associate to $D$ an ice quiver $(Q_D,F_D)$ as follows. The vertices of $Q_D$ are the alternating regions of $D$, as in Definition~\ref{d:postnikov-diag}. When the closures of two different alternating regions $v_1$ and $v_2$ meet in a crossing point between strands of $D$, or at one of the boundary marked points, we draw an arrow between $v_1$ and $v_2$, oriented in a way consistent with these strands, as shown in Figure~\ref{fig:dimer-arrows}.
\begin{figure}[h]
\begin{tikzpicture}
\draw[dashed] plot[smooth]
coordinates {(-1,-1) (-0.7,-0.5) (0,0) (0.7,0.5) (1,1)}
[postaction=decorate, decoration={markings,
 mark= at position 0.2 with \strarrow,
 mark= at position 0.85 with \strarrow}];
 
\draw[dashed] plot[smooth]
coordinates {(1,-1) (0.7,-0.5) (0,0) (-0.7,0.5) (-1,1)}
[postaction=decorate, decoration={markings,
 mark= at position 0.2 with \strarrow,
 mark= at position 0.85 with \strarrow}];
 
\node at (0,-1) (R1) {$v_1$};
\node at (0,1) (R2) {$v_2$};

\draw[quivarrow] (R1) -- (R2);
\end{tikzpicture}
\qquad\qquad
\begin{tikzpicture}
\path[ultra thick,\bdrycolor]
	(-1.5,0) edge (1.5,0);
\draw[dashed] plot[smooth]
coordinates {(1,2) (0.8,1) (0.5,0.4) (0,0)}
[postaction=decorate, decoration={markings,
 mark= at position 0.43 with \strarrow}];
 
\draw[dashed] plot[smooth]
coordinates {(0,0) (-0.5,0.4) (-0.8,1) (-1,2)}
[postaction=decorate, decoration={markings,
 mark= at position 0.63 with \strarrow}];
 
\node at (1.25,0.48) (R1) {$v_1$};
\node at (-1.25,0.48) (R2) {$v_2$};

\draw[frozarrow] (R1) edge[bend left] (R2);
\end{tikzpicture}
\caption{Arrows associated to strand crossings.}
\label{fig:dimer-arrows}
\end{figure}
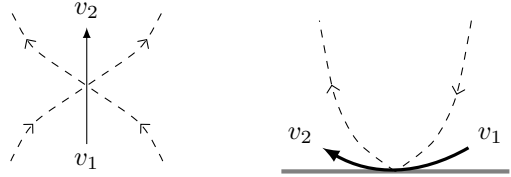
The vertices of $F_D$ are the boundary regions, and the arrows of $F_D$ are those corresponding to boundary marked points, as on the right of Figure~\ref{fig:dimer-arrows}.

Each clockwise or anticlockwise region $f$ of $D$ determines a \emph{fundamental cycle} $C_f$ in $Q$, by following the arrows corresponding to the crossing points in the boundary of the region, as shown in Figure~\ref{fig:fund-cycle} for a quadrilateral anticlockwise region.
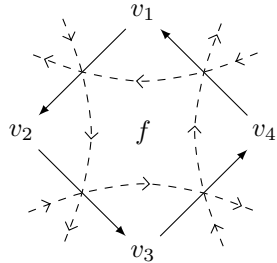
\begin{figure}[h]
\begin{tikzpicture}[scale=0.8]
\draw[dashed] plot[smooth]
coordinates {(35:2.3) (45:1.42) (90:0.85) (135:1.42) (145:2.3)}
[postaction=decorate, decoration={markings,
 mark= at position 0.1 with \strarrow,
 mark= at position 0.53 with \strarrow,
 mark= at position 0.93 with \strarrow}];
 
\draw[dashed] plot[smooth]
coordinates {(125:2.3) (135:1.42) (180:0.85) (225:1.42) (235:2.3)}
[postaction=decorate, decoration={markings,
 mark= at position 0.1 with \strarrow,
 mark= at position 0.53 with \strarrow,
 mark= at position 0.93 with \strarrow}];
 
\draw[dashed] plot[smooth]
coordinates {(215:2.3) (225:1.42) (270:0.85) (315:1.42) (325:2.3)}
[postaction=decorate, decoration={markings,
 mark= at position 0.1 with \strarrow,
 mark= at position 0.53 with \strarrow,
 mark= at position 0.93 with \strarrow}];
 
\draw[dashed] plot[smooth]
 coordinates {(305:2.3) (315:1.42) (0:0.85) (45:1.42) (55:2.3)}
[postaction=decorate, decoration={markings,
 mark= at position 0.1 with \strarrow,
 mark= at position 0.53 with \strarrow,
 mark= at position 0.93 with \strarrow}];

\node at (0,0) (f) {$f$};
\foreach \n in {1,2,3,4}
{\node at (90*\n:2) (R\n) {$v_\n$};}
\draw[quivarrow]
(R1) -- (R2);
\draw[quivarrow]
(R2) -- (R3);
\draw[quivarrow]
(R3) -- (R4);
\draw[quivarrow]
(R4) -- (R1);
\end{tikzpicture}
\caption{A fundamental cycle around an oriented region.}
\label{fig:fund-cycle}
\end{figure}

Thus we may define a potential
\[W_D=\sum_{\text{$f$ anticlockwise}}C_f-\sum_{\text{$f$ clockwise}}C_f\]
on $Q_D$, and we call $A_D:=\frjac{Q_D}{F_D}{W_D}$ the \emph{dimer algebra} of $D$.
\end{defn}

\begin{rem}
While it is straightforward to describe the defining relations of $A_D$ directly without using a potential (see e.g.\ \cite[Def.~2.11]{canakciperfect}), results on frozen Jacobian algebras will play a crucial role for us in Section~\ref{s:calabi-yau}, making it important that $A_D$ has this structure.
\end{rem}

\begin{prop}[cf.\ {\cite[Rem.~3.3]{baurdimer}}]
\label{p:strongly-connected}
The quiver $Q_D$ is strongly connected, i.e.\ any two vertices are connected by a directed path.
\end{prop}
\begin{proof}
Since the regions of $D$ cover the disc, the underlying graph of $Q_D$ is connected. Moreover, every arrow $a$ of $Q_D$ is contained in at least one fundamental cycle---exactly one if $a$ is frozen, and exactly two otherwise. Thus every arrow has a path from its head to its tail, and so $Q_D$ is even strongly connected.
\end{proof}

\begin{rem}
\label{r:positroids}
Having now defined Postnikov diagrams and their associated quivers and algebras more precisely, we may provide a few more details on the relationship to positroid varieties discussed in \S\ref{ss:positroids}.

Label the boundary marked points of $\Sigma$ by $\{1,\dotsc,n\}$, in the clockwise direction. These labels can be transferred first to the strands, by giving a strand the label of its source, and then to the alternating regions, by applying the label of a strand to every alternating region on its left. By the definition of $k$, this process labels each alternating region $v$ of $D$ by a \emph{$k$-label} $I_v\in\subsets{k}{n}$, and it turns out \cite[Thm.~6.6]{ohweak} that the collection $\clust_D$ of these labels is a maximal weakly separated collection in the positroid $\posit_D$. We note here that $\posit_D=\subsets{k}{n}$ (i.e.\ $\posit_D$ is the uniform positroid) if and only if $D$ is a $(k,n)$-diagram.

The cluster algebra $\clustalg{D}$ has initial cluster variables $x_v$ indexed by the vertices of $Q_D$. By construction, these vertices are the alternating regions of $D$, and so each $v$ has an attached $k$-label $I_v\in\subsets{k}{n}$ and Plücker coordinate $\Plueck{I_v}$. The assignment $x_v\mapsto\Plueck{I_v}$ induces Galashin and Lam's isomorphism $\clustalg{D}\to\CC[\coneopenposvar(\posit_D)]$ \cite{galashinpositroid}. (Strictly speaking, Galashin and Lam use the complementary set of $(n-k)$-labels obtained by labelling regions to the right of each strand, so they put a cluster structure on the coordinate ring of an isomorphic positroid variety in $\Gra{n-k}{n}$. Our isomorphism is obtained from theirs using the natural isomorphism $\CC[\Gra{n-k}{n}]\isoto\CC[\Gra{k}{n}]$ with $\Plueck{I}\mapsto\Plueck{I^c}$.)

We note that we could have obtained a different set of $k$-labels, and hence a different specialisation map $\clustalg{D}\to\CC(\coneopenposvar(\posit_D))$, by instead giving the strands of $D$ the labels of their targets. This can produce a different cluster algebra structure (in the sense that different functions are cluster variables) on $\CC[\coneopenposvar(\posit_D)]$, isomorphic to the same abstract cluster algebra $\clustalg{D}$. While this paper is mostly concerned with categorifying $\clustalg{D}$, and so is insensitive to such labelling conventions, they will reappear in Section~\ref{s:Grassmannian}, when we use results of the author with Çanakçı and King \cite{canakciperfect} to relate our categories more directly with the coordinate rings $\CC[\coneopenposvar(\posit_D)]$. In particular, it will turn out to be the source labelling convention that is most compatible with our categorification and its relationship to Jensen--King--Su's Grassmannian cluster category $\CM(C)$ \cite{jensencategorification}.
\end{rem}

\begin{rem}
Given a Postnikov diagram $D$ in the disc $\Sigma$, we can construct from $D$ a bipartite graph $G$ as follows. The vertices of $G$ are given by the anticlockwise and clockwise regions of $D$, and two such vertices are connected by an edge if the the closures of the regions meet at a crossing of strands in the interior of $\Sigma$. The graph $G$ is connected if and only if $D$ is connected. Adding to $G$ the data of half-edges, which connect the vertex corresponding to a region to any boundary marked points in its closure, we obtain a dimer model in $\Sigma$ in the sense of \cite[Ex.~2.5]{presslandmutation}. This construction appears in \cite[\S14]{postnikovtotal}, where the dimer model is called a plabic graph. The dimer algebra $A_D$ is precisely the dimer algebra of $(\Sigma,G)$, as defined in \cite[Ex.~2.12]{presslandmutation}. An example of a Postnikov diagram, its corresponding dimer model, and the ice quiver of its dimer algebra, is shown in Figure~\ref{f:dimereg}.

\begin{figure}[h]
    \centering
    \begin{minipage}[t]{0.33\textwidth}
        \centering
\begin{tikzpicture}[scale=2,baseline=(bb.base)]

\path (0,0) node (bb) {};

\draw [boundary] (0,0) circle(1.0);

\foreach \n/\a in {1/-12, 2/0, 3/5, 4/10, 5/12}
{ \coordinate (b\n) at (\bstart-\fifth*\n+\a:1.0);
\coordinate (l\n) at (\bstart-\fifth*\n+\a-4:0.85);
\coordinate (r\n) at (\bstart-\fifth*\n+\a+4:0.85);}

\foreach \n/\m in {6/1, 7/2, 8/3, 9/4, 10/5}
  {\coordinate (b\n) at ($0.55*(b\m)$);}
\foreach \n/\m in {12/2, 13/4}
  {\coordinate (b\n) at ($0.80*(b\m)$);}

\coordinate (b11) at ($0.4*(b6) + 0.4*(b10)$);

\foreach \n/\s in {6/1.2, 7/0.9, 8/0.7, 9/0.9, 10/1.2}
{\coordinate (b\n) at ($\s*(b\n)$);}



\foreach \e/\f/\t in {1/6/0.5, 2/12/0.5, 7/12/0.5, 3/8/0.5, 4/13/0.5, 9/13/0.5, 5/10/0.5, 6/7/0.5, 7/8/0.5, 8/9/0.5, 9/10/0.5, 10/11/0.5, 6/11/0.5, 8/11/0.6}
{\coordinate (a\e-\f) at ($(b\e) ! \t ! (b\f)$); }

\draw [strand] plot[smooth]
coordinates {(b1) (a6-11) (a8-11) (a8-9) (a9-13) (l4) (b4)}
[postaction=decorate, decoration={markings,
 mark= at position 0.16 with \strarrow,
  mark= at position 0.37 with \strarrow,
 mark= at position 0.59 with \strarrow,
 mark= at position 0.76 with \strarrow,
  mark= at position 0.92 with \strarrow }];
 
\draw [strand] plot[smooth]
coordinates {(b2) (r2) (a7-12) (a7-8) (a8-11) (a10-11) (b5)}
[postaction=decorate, decoration={markings,
 mark= at position 0.1 with \strarrow,
 mark= at position 0.26 with \strarrow,
 mark= at position 0.44 with \strarrow,
 mark= at position 0.64 with \strarrow,
 mark= at position 0.85 with \strarrow }];
 
\draw [strand] plot[smooth]
coordinates {(b3) (a7-8) (a6-7) (b1)}
[postaction=decorate, decoration={markings,
 mark= at position 0.24 with \strarrow,
 mark= at position 0.54 with \strarrow, 
 mark= at position 0.82 with \strarrow }];

\draw [strand] plot[smooth]
coordinates {(b4) (r4) (a9-13) (a9-10) (a10-11) (a6-11) (a6-7) (a7-12) (l2) (b2)}
 [postaction=decorate, decoration={markings,
 mark= at position 0.08 with \strarrow,
 mark= at position 0.2 with \strarrow,
  mark= at position 0.36 with \strarrow,
 mark= at position 0.52 with \strarrow, 
 mark= at position 0.67 with \strarrow,
  mark= at position 0.81 with \strarrow,
 mark= at position 0.93 with \strarrow }];

\draw [strand] plot[smooth]
coordinates {(b5) (a9-10) (a8-9) (b3)}
[postaction=decorate, decoration={markings,
 mark= at position 0.2 with \strarrow,
 mark= at position 0.5 with \strarrow, 
 mark= at position 0.78 with \strarrow }];

%

%

 \end{tikzpicture}
    \end{minipage}
    \begin{minipage}[t]{0.33\textwidth}
        \centering
\begin{tikzpicture}[scale=2,baseline=(bb.base)]

\path (0,0) node (bb) {};

\draw [boundary] (0,0) circle(1.0);

\foreach \n/\a in {1/-12, 2/0, 3/5, 4/10, 5/12}
{ \coordinate (b\n) at (\bstart-\fifth*\n+\a:1.0);
\coordinate (l\n) at (\bstart-\fifth*\n+\a-4:0.85);
\coordinate (r\n) at (\bstart-\fifth*\n+\a+4:0.85);}

\foreach \n/\m in {6/1, 7/2, 8/3, 9/4, 10/5}
  {\coordinate (b\n) at ($0.55*(b\m)$);}
\foreach \n/\m in {12/2, 13/4}
  {\coordinate (b\n) at ($0.80*(b\m)$);}

\coordinate (b11) at ($0.4*(b6) + 0.4*(b10)$);

\foreach \n/\s in {6/1.2, 7/0.9, 8/0.7, 9/0.9, 10/1.2}
{\coordinate (b\n) at ($\s*(b\n)$);}

\foreach \h/\t in {1/6, 2/12, 7/12, 3/8, 4/13, 9/13, 5/10, 6/7, 7/8, 8/9, 9/10, 10/11, 6/11, 8/11}
{ \draw [bipedge] (b\h)--(b\t); }

\foreach \n in {6,8,10,12,13} 
  {\draw [\graphcolor] (b\n) circle(\dotrad) [fill=\graphcolor];} 
\foreach \n in {7,9,11}  
  {\draw [\graphcolor] (b\n) circle(\dotrad) [fill=white];} 

\foreach \e/\f/\t in {1/6/0.5, 2/12/0.5, 7/12/0.5, 3/8/0.5, 4/13/0.5, 9/13/0.5, 5/10/0.5, 6/7/0.5, 7/8/0.5, 8/9/0.5, 9/10/0.5, 10/11/0.5, 6/11/0.5, 8/11/0.6}
{\coordinate (a\e-\f) at ($(b\e) ! \t ! (b\f)$); }

 \end{tikzpicture}
    \end{minipage}\hfill
    \begin{minipage}[t]{0.33\textwidth}
        \centering
\begin{tikzpicture}[scale=2,baseline=(bb.base)]

\path (0,0) node (bb) {};


\foreach \n/\a in {1/-12, 2/0, 3/5, 4/10, 5/12}
{ \coordinate (b\n) at (\bstart-\fifth*\n+\a:1.0);
  \coordinate (l\n) at (\bstart-\fifth*\n+\a-4:0.85);
  \coordinate (r\n) at (\bstart-\fifth*\n+\a+4:0.85);}

\foreach \n/\m in {6/1, 7/2, 8/3, 9/4, 10/5}
  {\coordinate (b\n) at ($0.55*(b\m)$);}
\foreach \n/\m in {12/2, 13/4}
  {\coordinate (b\n) at ($0.80*(b\m)$);}

\coordinate (b11) at ($0.4*(b6) + 0.4*(b10)$);

\foreach \n/\s in {6/1.2, 7/0.9, 8/0.7, 9/0.9, 10/1.2}
{\coordinate (b\n) at ($\s*(b\n)$);}



\foreach \e/\f/\t in {1/6/0.5, 2/12/0.5, 7/12/0.5, 3/8/0.5, 4/13/0.5, 9/13/0.5, 5/10/0.5, 6/7/0.5, 7/8/0.5, 8/9/0.5, 9/10/0.5, 10/11/0.5, 6/11/0.5, 8/11/0.6}
{\coordinate (a\e-\f) at ($(b\e) ! \t ! (b\f)$); }

\foreach \n/\m/\a in {1/45/0, 2/15/-7, 3/12/0, 4/23/10, 5/34/12}
{ \draw (\bstart+\fifth/2-\fifth*\n+\a:0.93) node (q\m) {$\diamond$}; }

\foreach \m/\a/\r in {25/14/0.25, 24/176/0.22}
{ \draw (\a:\r) node (q\m) {$\bullet$}; }

\foreach \t/\h/\a in {23/12/-27, 34/23/-27, 45/34/-32, 15/45/-32, 12/15/-27}
{ \draw [frozarrow]  (q\t) edge [bend left=\a] (q\h); }
 
\foreach \t/\h/\a in {25/15/12, 15/12/-25, 12/25/23, 25/24/2, 24/45/20, 45/25/20, 24/23/20, 23/34/-20, 34/24/12}
{ \draw [quivarrow]  (q\t) edge [bend left=\a] (q\h); }

 \end{tikzpicture}
    \end{minipage}
\caption{A Postnikov diagram (left), its corresponding dimer model (centre) and the ice quiver of its dimer algebra (right). The frozen vertices of this ice quiver are shown as white diamonds, and the frozen arrows are drawn in bold.}
\label{f:dimereg}
\end{figure}
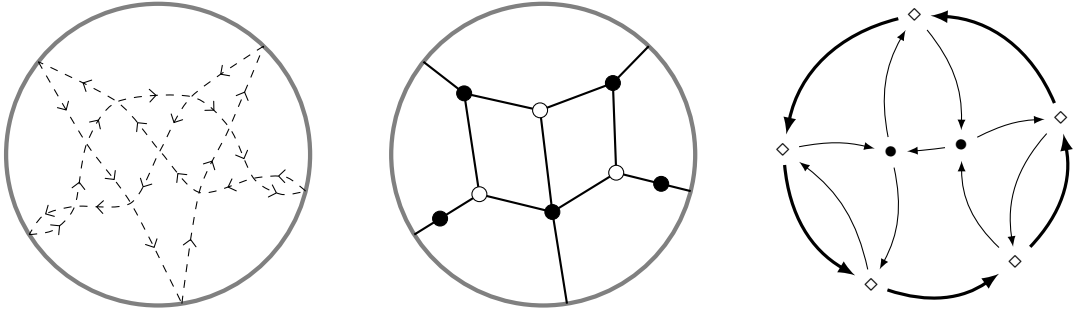

Given the data of a dimer model $G$ in the disc, or indeed in any oriented surface $\Sigma$ with or without boundary, it is always possible to produce a collection of strands, dividing $\Sigma$ into alternating, clockwise and anticlockwise regions, from which $G$ may be recovered by the above rules. These strands are usually called the \emph{zig-zag paths} or flows of $G$ \cite[\S3.3]{broomheaddimer}. Note that in general these strands can either travel between points on the boundary of $\Sigma$, as in Definition~\ref{d:postnikov-diag}, or be closed curves in the interior. In the case that $\Sigma$ is a disc, these strands will always satisfy conditions (P0)--(P2) from Definition~\ref{d:postnikov-diag}. Requiring that they also satisfy (P3) and (P4)---which in particular rules out any closed curves in the interior---places an extra condition on the dimer model $G$, the analogue of Broomhead's geometric consistency \cite[Prop.~3.12]{broomheaddimer}. Since these conditions, particularly (P4), will be important for us, we prefer to start from the data of a Postnikov diagram.
\end{rem}

Postnikov diagrams are typically considered up to isotopy fixing the boundary, which does not affect the construction of the dimer algebra at all, and certain twisting and untwisting moves shown in Figure~\ref{f:untwist}. These moves do affect the construction of the dimer algebra, in the sense that two Postnikov diagrams related by such a move determine different ice quivers with potential, as shown in Figure~\ref{f:untwist-quiv} (cf.~\cite[Ex.~4.5]{presslandmutation}). However, the frozen Jacobian algebras of these two quivers with potential, i.e.\ the dimer algebras of the two Postnikov diagrams, are isomorphic in a natural way. These moves also do not affect the $k$-labels attached to the quiver vertices.
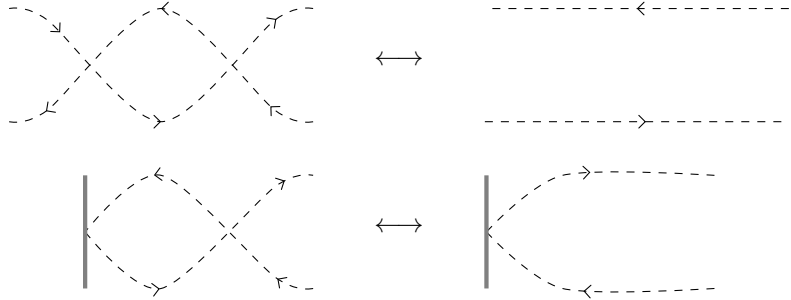
\begin{figure}[h]
$\mathord{\begin{tikzpicture}[baseline=0,yscale=0.75]
\draw[dashed] plot[smooth]
coordinates {(-2,1) (-1.5,0.8) (0,-1) (1.5,0.8) (2,1)}
[postaction=decorate, decoration={markings,
 mark= at position 0.15 with \strarrow,
 mark= at position 0.5 with \strarrow,
 mark= at position 0.9 with \strarrow}];
 
\draw[dashed] plot[smooth]
coordinates {(2,-1) (1.5,-0.8) (0,1) (-1.5,-0.8) (-2,-1)}
[postaction=decorate, decoration={markings,
 mark= at position 0.12 with \strarrow,
 mark= at position 0.5 with \strarrow,
 mark= at position 0.9 with \strarrow}];
\end{tikzpicture}}
\qquad\longleftrightarrow\qquad
\mathord{\begin{tikzpicture}[baseline=0,yscale=0.75]
\draw[dashed] plot[smooth]
coordinates {(2,1) (-2,1)}
[postaction=decorate, decoration={markings,
 mark= at position 0.5 with \strarrow}];
 
\draw[dashed] plot[smooth]
coordinates {(-2,-1) (2,-1)}
[postaction=decorate, decoration={markings,
 mark= at position 0.53 with \strarrow}];
\end{tikzpicture}}$\vspace{0.5cm}

$\mathord{\begin{tikzpicture}[baseline=0,yscale=0.75]
\draw[dashed] plot[smooth]
coordinates {(-1,0) (0,-1) (1.5,0.8) (2,1)}
[postaction=decorate, decoration={markings,
 mark= at position 0.33 with \strarrow,
 mark= at position 0.9 with \strarrow}];
 
\draw[dashed] plot[smooth]
coordinates {(2,-1) (1.5,-0.8) (0,1) (-1,0)}
[postaction=decorate, decoration={markings,
 mark= at position 0.13 with \strarrow,
 mark= at position 0.69 with \strarrow}];
 
\draw [boundary] (-1,1) -- (-1,-1);
\end{tikzpicture}}
\qquad\longleftrightarrow\qquad
\mathord{\begin{tikzpicture}[baseline=0,yscale=0.75]
\draw[dashed] plot[smooth]
coordinates {(-1,0) (0,1) (2,1)}
[postaction=decorate, decoration={post length=1mm,pre length=1mm,markings,
 mark= at position 0.5 with \strarrow}];
 
\draw[dashed] plot[smooth]
coordinates {(2,-1) (0,-1) (-1,0)}
[postaction=decorate, decoration={post length=1mm,pre length=1mm,markings,
 mark= at position 0.53 with \strarrow}];
 
\draw [boundary] (-1,1) -- (-1,-1);
\end{tikzpicture}}$
\caption{Twisting (right-to-left) and untwisting (left-to-right) moves for a Postnikov diagram, in the interior (above) and at the boundary (below). The reflections of these figures in a horizontal line also show twisting and untwisting moves.}
\label{f:untwist}
\end{figure}
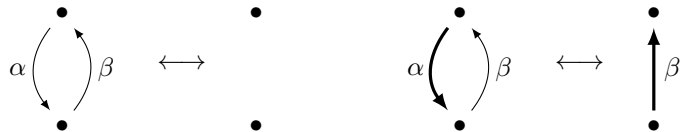
\begin{figure}[h]
\begin{minipage}[t]{0.35\textwidth}
\centering
$\mathord{\begin{tikzpicture}[baseline=0,yscale=0.75]
\draw (0,1) node (v1) {$\bullet$};
\draw (0,-1) node (v2) {$\bullet$};

\draw [quivarrow] (v1) edge [bend right] node[midway,xshift=-6] {$a$} (v2) ;
\draw [quivarrow] (v2) edge [bend right] node[midway,xshift=6] {$b$}(v1);
\end{tikzpicture}}
\quad\longleftrightarrow\quad
\mathord{\begin{tikzpicture}[baseline=0,yscale=0.75]
\draw (0,1) node (v1) {$\bullet$};
\draw (0,-1) node (v2) {$\bullet$};
\end{tikzpicture}}$
\end{minipage}
\begin{minipage}[t]{0.35\textwidth}
\centering
$\mathord{\begin{tikzpicture}[baseline=0,yscale=0.75]
\draw (0,1) node (v1) {$\bullet$};
\draw (0,-1) node (v2) {$\bullet$};

\draw [frozarrow] (v1) edge [bend right] node[midway,xshift=-6] {$a$} (v2) ;
\draw [quivarrow] (v2) edge [bend right] node[midway,xshift=6] {$b$}(v1);
\end{tikzpicture}}
\quad\longleftrightarrow\quad
\mathord{\begin{tikzpicture}[baseline=0,yscale=0.75]
\draw (0,1) node (v1) {$\bullet$};
\draw (0,-1) node (v2) {$\bullet$};

\draw [frozarrow] (v2) edge node[midway,right] {$b$}(v1);
\end{tikzpicture}}$
\end{minipage}
\caption{The effect on the quiver of applying a twisting or untwisting move in the interior (left) or at the boundary (right). Bold arrows are frozen.}
\label{f:untwist-quiv}
\end{figure}

\begin{prop}[cf.~{\cite[Lem.~12.1]{baurdimer}}]
\label{p:untwisting}
Let $D$ and $D'$ be Postnikov diagrams such that $D'$ is obtained from $D$ via an untwisting move. Then the dimer algebras $A_D$ and $A_{D'}$ are isomorphic.
\end{prop}

\begin{proof}
First we consider the case that the untwisting move happens in the interior, as in the upper part of Figure~\ref{f:untwist}. Then $Q_{D'}$ is obtained from $Q_D$ by removing two arrows $a$ and $b$ which form a fundamental cycle, as in the left-hand side of Figure~\ref{f:untwist-quiv}. Since these arrows are unfrozen, they are each contained in a second fundamental cycle, and we write these cycles as $ap$ and $bq$ for some paths $p$ and $q$ respectively. Since these paths do not contain the arrows $a$ or $b$, they are also paths in $Q_{D'}$. Moreover, the relation $\der{a}{W_D}=0$ implies that $b=p$ in $A_D$, and the relation $\der{b}{W_D}=0$ that $a=q$ in $A_D$. Define a map $\varphi\colon\cpa{\KK}{Q_D}\to\cpa{\KK}{Q_{D'}}$ fixing all vertex idempotents (the two quivers having the same vertex set) and all arrows different from $a$ and $b$, and with $\varphi(a)=q$ and $\varphi(b)=p$. This map induces a homomorphism $A_D\to A_{D'}$ inverse to that induced by the map $\cpa{\KK}{Q_{D'}}\to\cpa{\KK}{Q_D}$ arising from the inclusion map between arrow sets.

On the other hand, if the untwisting move takes place at the boundary, then $Q_{D'}$ is obtained from $Q_D$ by removing a frozen arrow $a$ lying in a fundamental $2$-cycle, and $F_{D'}$ is obtained from $F_D$ by replacing $a$ by the other arrow $b$ in this $2$-cycle. As before, since $b$ is unfrozen it lies in a second fundamental cycle $bq$ for some path $q$ not involving $a$, and it follows from the relation $\der{b}{W_D}=0$ that $a=q$ in $A_D$. Thus defining $\varphi$ as in the first case except for $\varphi(b)=b$ provides an isomorphism of dimer algebras in an analogous way.
\end{proof}

We call a Postnikov diagram \emph{reduced} if no untwisting moves can be applied to it; any Postnikov diagram $D$ is equivalent under untwisting moves to a reduced diagram $D'$, unique up to isotopy fixing the boundary. The reader is warned that this does not correspond to Postnikov's definition of reducedness for plabic graphs \cite[\S12]{postnikovtotal}, but rather is compatible with terminology for quivers with potential: the ice quiver with potential $(Q_{D'},F_{D'},W_{D'})$ is the reduction of $(Q_D,F_D,W_D)$, in the sense of \cite[\S3]{presslandmutation}. Indeed, the above isomorphisms of dimer algebras in the proof of Proposition~\ref{p:untwisting} are local versions of the reduction isomorphism given at the end of the proof of \cite[Thm.~3.6]{presslandmutation}.

\begin{prop}
\label{p:reduced-quiv}
If $D$ is reduced, the quiver $Q_D$ has no $2$-cycles.
\end{prop}
\begin{proof}
A $2$-cycle in $Q_D$ containing at least one unfrozen arrow corresponds to a possible untwisting move (see Figures~\ref{f:untwist} and \ref{f:untwist-quiv}), and so do not occur when $D$ is reduced. A $2$-cycle consisting only of frozen arrows would correspond to a pair of boundary regions meeting at two points on the boundary of the disc, which must then be the only boundary regions. But $D$ has at least $3$ boundary regions, so this is also impossible.
\end{proof}

Before moving on to the homological part of the paper, we discuss one additional feature of the dimer algebra $A_D$ of a connected Postnikov diagram $D$. For each vertex $v$ of $Q_D$, choose a path $t_v\colon v\to v$ representing a fundamental cycle. Writing $t=\sum_{v\in Q_0}t_v\in A_D$, it follows from connectedness of $D$ that $t$ does not depend on the choices of paths $t_v$. Moreover, $t$ is a central element of the algebra. Thus $A_D$ has the structure of an algebra over $Z=\powser{\KK}{t}$; the abuse of notation in identifying the abstract generator of $Z$ with the element $t\in A_D$, thus making the $Z$-action clear, is justified by the following result, which is also fundamental to the arguments in Section~\ref{s:calabi-yau}. The proof is close to that of Baur--King--Marsh \cite{baurdimer} for the case that $D$ is a $(k,n)$-diagram, and we refer to their paper when the arguments apply without change to our more general setting.

\begin{prop}
\label{p:thin}
The dimer algebra $A_D$ of a connected Postnikov diagram $D$ is thin, meaning that for any vertices $v$ and $w$ there is some $p\in\idemp{w}A_D\idemp{v}$ such that
\[\idemp{w}A_D\idemp{v}=Zp,\]
and moreover that $Zp$ is a free $Z$-module of rank $1$.
Moreover, there is a unique such $p$ expressible as the image of a path from $v$ to $w$ under the projection $\cpa{\KK}{Q_D}\to A_D$.
\end{prop}

\begin{proof}
As in \cite[\S4]{baurdimer}, we may weight the arrows of $Q=Q_D$ by elements of the set $\set{1,\dotsc,n}$ labelling the boundary marked points of $\Sigma$. An arrow $a$ of $Q$ is crossed by two strands of $D$, say that starting at marked point $i$ from right to left, and that starting at $j$ from left to right. We then weight $a$ by the (indicator function of) the cyclic interval $[i,j-1]$ (cf.\ \cite[Defn.~4.1]{baurdimer}). A path in $Q$ is weighted by the sum $\wt_p$ of weights of its arrows, and its total weight is defined to be $\sum_{i=1}^n\wt_p(i)$, which is always at least $1$ if the path has non-zero length.

The proof of \cite[Cor.~4.4]{baurdimer}, stated for $(k,n)$-Postnikov diagrams, remains valid in our more general setting to show that every fundamental cycle has constant weight $w(i)=1$ for all $1\leq i\leq n$. If $p_+=p_-$ is one of the defining relations of $A_D$, then there is an arrow $a\in Q_1$ such that both $a p_+$ and $a p_-$ are fundamental cycles, from which it follows that the weights of $p_+$ and $p_-$ agree. It follows that the weight, and hence the total weight, descends to a grading of $A$.

Now let $v,w\in Q_0$. Since $Q$ is strongly connected as in Proposition~\ref{p:strongly-connected}, there is some path from $v$ to $w$ in $Q$, and we choose $p$ to be such a path with minimal total weight. If $q$ is any other path from $v$ to $w$, then \cite[Prop.~9.3]{baurdimer} applies to show that there is a path $r\colon v\to w$ and non-negative integers $N_p$ and $N_q$ such that
\[p=t^{N_p}r,\qquad q=t^{N_q}r\]
in $A_D$. As before, this proposition is stated only in the case that $D$ is a $(k,n)$-Postnikov diagram, but its proof is still valid under our weaker assumptions---the key property of $D$ here is (P4).

Since the total weight of $t$ is non-zero, and $p$ has minimal total weight among paths from $v$ to $w$, we must have $N_p=0$ and $p=r$ in $A_D$. Thus $q=t^{N_q}p$, and we see that $e_wAe_v$ is generated over $Z$ by $p$. It is moreover freely generated since each element of $\{t^Np:N\geq0\}$ has a different total weight, which implies that these elements are linearly independent in $A_D$.
\end{proof}

\begin{defn}
\label{d:minpath}
We call a path $p\colon v\to w$ in $Q_D$ a \emph{minimal path} if $\idemp{w}A_D\idemp{v}=Zp$. Proposition~\ref{p:thin} then states that there is a minimal path between any pair of vertices in $Q_D$, and that any two such paths define the same element of $A_D$. We write $\minpath{v}{w}$ for the class in $A_D$ of any minimal path from $v$ to $w$.
\end{defn}

Thinness of $A_D$, as in Proposition~\ref{p:thin}, is the analogue of algebraic consistency for dimer models on  the torus, as described by Broomhead \cite[Def.~5.12]{broomheaddimer} (the analogue of geometric consistency, cf.\ \cite[Prop.~3.12]{broomheaddimer}, being conditions (P3) and (P4) in the definition of a Postnikov diagram). The following observation about minimal paths is straightforward but useful.

\begin{lem}
\label{l:min-subpath}
Let $D$ be a connected Postnikov diagram, and let $p\colon u\to v$ and $q\colon v\to w$ be paths in $Q_D$. If the composition $qp$ is a minimal path, then both $p$ and $q$ are minimal paths.
\end{lem}
\begin{proof}
By Proposition~\ref{p:thin}, there are $m,n\geq0$ such that $p=t^m\minpath{u}{v}$ and $q=t^n\minpath{v}{w}$ in $A_D$. Thus $qp=t^{m+n}\minpath{v}{w}\minpath{u}{v}=t^{m+n+\delta}\minpath{u}{w}$ for some $\delta\geq0$. But $qp$ is minimal, so by Proposition~\ref{p:thin} again we must have $m+n+\delta=0$, and so in particular $m=n=0$. Hence both $p$ and $q$ are minimal paths.
\end{proof}

\section{\texorpdfstring{The Calabi--Yau property}{The Calabi–Yau property}}
\label{s:calabi-yau}

Let $A=A_D$ be the dimer algebra of a connected Postnikov diagram $D$. In this section we show that this algebra is bimodule internally $3$-Calabi--Yau, in the sense of \cite{presslandinternally}, with respect to the idempotent corresponding to the boundary regions of $D$, thus proving Theorem~\ref{t:cy-intro}. Since $A$ is presented as a frozen Jacobian algebra, we will prove this statement via \cite[Thm.~5.6]{presslandinternally}, which tells us that it is enough to check the exactness of a certain complex of $A$-bimodules.

Throughout the proof, we will write
\[\Head{v}=\set{a\in Q_1:\head{a}=v},\qquad\Tail{v}=\set{a\in Q_1:\tail{a}=v}\]
for the set of arrows with head, respectively tail, at $v\in Q_0$, and
\[\mHead{v}=\Head{v}\cap Q_1^{\mut},\qquad\mTail{v}=\Tail{v}\cap Q_1^{\mut}\]
for their intersections with the set $Q_1^\mut$ of unfrozen arrows.

As well as using the fact that $A$ is thin, as in Proposition~\ref{p:thin}, we will also rely on the following gradability property of $A$, which allows us to deduce exactness of the relevant complex of bimodules from the exactness of certain complexes of (one-sided) $A$-modules.

\begin{prop}
\label{p:grading}
The dimer algebra $A_D$ of a connected Postnikov diagram $D$ admits a grading in which all arrows have positive degrees.
\end{prop}

\begin{proof}
The algebra $A_D$ may be graded by total weights of paths, as discussed in the proof of Proposition~\ref{p:thin}. As already mentioned in that proof, each arrow has positive total weight.
\end{proof}
%
%

Most of the arguments in this section will depend on the fact that $A$ is thin, as in Proposition~\ref{p:thin}, so that each piece $\idemp{w}A\idemp{v}$ is freely generated over $Z=\powser{\KK}{t}$ by a minimal path. Before introducing the complexes of $A$-modules whose exactness implies the required $3$-Calabi--Yau property of $A$, we give one more lemma, concerning these paths.

\begin{lem}
\label{l:nb-lemma}
Let $v,w\in Q_0$ with $v$ mutable, and consider a minimal path $\minpath{v}{w}\colon v\to w$. Then there is some $a\in\Head{v}$ such that $\minpath{v}{w}a=\minpath{,\tail{a}}{w}$ in $A$.
\end{lem}
\begin{proof}
We first observe that if $p\colon u\to w$ is a minimal path in $Q$ passing through the vertex $v$, then $p$ includes our desired arrow as follows. Write $p=p_2ap_1$ for some $a\in\Head{v}$. Then by Lemma~\ref{l:min-subpath} each subpath of $p$ is also minimal, so in $A$ we have $p_2=\minpath{v}{w}$ and $p_2a=\minpath{,\tail{a}}{w}$, and hence $a$ is our desired arrow.

Now for each $a\in\Head{v}$, choose a minimal path $m_a\colon\tail{a}\to w$. If any of these paths passes through $v$, then we find our desired arrow as above, so we assume the contrary. Since $v$ is mutable, each $b\in\Tail{v}$ is involved in two fundamental cycles $a_+pb$ and $a_-qb$ with $a_{\pm}\in\Head{v}$, and we can write $p'=m_{a_+}p$ and $q'=m_{a_-}q$. By our assumption on the $m_a$, neither $p'$ nor $q'$ passes through $v$.

Using again that $v$ is mutable, the union of fundamental cycles containing $v$ is a disc with $v$ in the interior, and so there must be some $b$ for which the paths $p'$ and $q'$ are two sides of a digon containing $v$ (cf.\ the argument in \cite[Lem.~6.17]{broomheaddimer}). From now on, we assume that we are in this situation, shown in Figure~\ref{f:digon-config}.

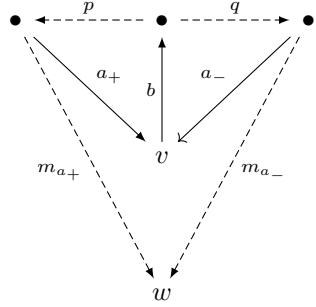
\begin{figure}[h]
\begin{tikzcd}[column sep=40pt,row sep=40pt]
\bullet\arrow[quivarrow,dashed]{ddr}[swap]{m_{a_+}}\arrow[quivarrow]{dr}{a_+}&\bullet\arrow[quivarrow,dashed]{l}[swap]{p}\arrow[quivarrow,dashed]{r}{q}&\bullet\arrow{dl}[swap]{a_-}\arrow[quivarrow,dashed]{ddl}{m_{a_-}}\\
&v\arrow[quivarrow]{u}{b}\\
&w
\end{tikzcd}
\caption{Fundamental cycles $a_+pb$ and $a_-qb$ together with minimal paths $m_{a_{\pm}}\colon\tail{a_{\pm}}\to w$ such that $p'=m_{a_+}p$ and $q'=m_{a_-}q$ form a digon containing $v$. Solid arrows represent arrows in $Q$, whereas dashed arrows represent paths.}
\label{f:digon-config}
\end{figure}

Observe that $\minpath{v}{w}a_+=t^{\delta}\minpath{,\tail{a_+}}{w}$ for some $\delta\geq0$ by Proposition~\ref{p:thin}. If $\delta=0$, then $a_+$ was our desired arrow, so assume $\delta>0$. In a similar way, we find $m,n\geq0$ such that $\minpath{,\tail{a_+}}{w}p=t^m\minpath{,\head{b}}{w}$ and $\minpath{,\head{b}}{w}b=t^n\minpath{v}{w}$. Since $pba_+$ is a fundamental cycle, we must have
\[t\minpath{,\tail{a_+}}{w}=\minpath{,\tail{a_+}}{w}pba_+=t^m\minpath{,\head{b}}{w}ba_+=t^{m+n}\minpath{v}{w}a_+=t^{m+n+\delta}\minpath{,\tail{a_+}}{w}.\]
By Proposition~\ref{p:thin} again, we must have $m+n+\delta=1$, and so we conclude from positivity of $\delta$ that $m=n=0$ (and $\delta=1$). In particular, this means that $p'=m_{a_+}p$ is a minimal path, representing $\minpath{,\head{b}}{w}$. Repeating the argument with $a_-$ and $q$ replacing $a_+$ and $p$ we see that either $a_-$ is our desired arrow or $q'=m_{a_-}q$ is also a minimal path representing $\minpath{,\head{b}}{w}$.

In the latter case, $p'$ and $q'$ are paths from $\head{b}$ to $w$ defining the same element of $A$, namely $\minpath{,\head{b}}{w}$, and bounding a digon containing $v$. It is then a consequence of \cite[Lem.~6.17]{broomheaddimer} that there is a path $r\colon\head{b}\to w$ passing through $v$, and also defining $\minpath{,\head{b}}{w}$ in $A$. In particular, $r$ is a minimal path passing through $v$, and so we may find our desired arrow as in the first paragraph of the proof.
\end{proof}

We now turn to the main part of the argument. In \cite[\S5]{presslandinternally}, it is explained how an ice quiver with potential $(Q,F,W)$ determines a complex of projective bimodules for the associated frozen Jacobian algebra $A=\frjac{Q}{F}{W}$; we will denote this complex by $\res{A}$, although strictly it depends on the presentation of $A$ determined by $(Q,F,W)$. For the convenience of the reader, we repeat its definition here, following \cite[\S4]{presslandcategorification}. 

Let $S=A/\close{J}$, where $J$ is the the ideal generated by arrows. As a left $A$-module, $S$ is the direct sum of the vertex simple modules, and has a basis consisting of the vertex idempotents $e_v$. For the remainder of this section, we write $\tens=\tens_S$.

Introduce formal symbols $\rho_a$ for each $a\in Q_1$ and $\omega_v$ for each $v\in Q_0$, and define $S$-bimodule structures on the vector spaces
\begin{align*}
\KK Q_0&=\bigdsum_{v\in Q_0}\KK\idemp{v},&
\KK Q_1&=\bigdsum_{a\in Q_1}\KK a,&
\KK Q_2^\mut&=\bigdsum_{a\in Q_1^\mut}\KK\rel{a},&
\KK Q_3^\mut&=\bigdsum_{v\in Q_0^\mut}\KK\omega_v,
\end{align*}
via the formulae
\[
\idemp{v}\cdot\idemp{v}\cdot\idemp{v}=\idemp{v},\qquad
\idemp{\head{a}}\cdot a\cdot\idemp{\tail{a}}=a,\qquad
\idemp{\tail{a}}\cdot\rel{a}\cdot\idemp{\head{a}}=\rel{a},\qquad
\idemp{v}\cdot\omega_v\cdot\idemp{v}=\omega_v.
\]
We recall here that $Q_0^\mut=Q_0\setminus F_0$ and $Q_1^\mut=Q_1\setminus F_1$ are the sets of mutable vertices and unfrozen arrows respectively.

Since $\KK Q_0$ is naturally isomorphic to $S$ as a bimodule, there is a natural isomorphism
\[A\tens\KK Q_0\tens A\isoto A\tens A,\]
which we can compose with the multiplication map for $A$ to obtain a map $\mu_0\colon A\tens\KK Q_0\tens A\to A$. Define $\mu_1\colon A\tens\KK Q_1\tens A\to A\tens\KK Q_0\tens A$ by
\[\mu_1(x\tens a\tens y)=x\tens\idemp{\head{a}}\tens a y-xa\tens\idemp{\tail{a}}\tens y.\]
For any path $p=a_m\dotsm a_1$ in $Q$, we may define
\[\Delta_a(p)=\sum_{a_i=a}a_m\dotsm a_{i+1}\tens a_i\tens a_{i-1}\dotsm a_1,\]
and extend by linearity and continuity to obtain a map $\Delta_a\colon\cpa{\KK}{Q}\to A\tens\KK Q_1\tens A$. We then define $\mu_2\colon A\tens\KK Q_2^{\mut}\tens A\to A\tens \KK Q_1\tens A$ by
\[\mu_2(x\tens\rel{a}\tens y)=\sum_{b\in Q_1}x\Delta_b(\der{a}{W})y.\]
Finally, define $\mu_3\colon A\tens \KK Q_3^{\mut}\tens A\to A\tens \KK Q_2^{\mut}\tens A$ by
\[\mu_3(x\tens\omega_v\tens y)=\sum_{a\in\Tail{v}}x\tens\rel{a}\tens a y-\sum_{b\in\Head{v}}xb\tens\rel{b}\tens y.\]
As $\Tail{v}\union\Head{v}\subseteq Q_1^\mut$ for any $v\in Q_0^\mut$, this map has the claimed codomain.
\begin{defn}
\label{frjacres}
For an ice quiver with potential $(Q,F,W)$, with frozen Jacobian algebra $A=\frjac{Q}{F}{W}$, let $\res{A}$ be the complex of $A$-bimodules with non-zero terms
\[\begin{tikzcd}
A\tens \KK Q_3^\mut\tens A\arrow{r}{\mu_3}&A\tens \KK Q_2^\mut\tens A\arrow{r}{\mu_2}&A\tens \KK Q_1\tens A\arrow{r}{\mu_1}&A\tens\KK Q_0\tens A
\end{tikzcd}\]
and $A\tens\KK Q_0\tens A$ in degree $0$.
\end{defn}

Using this complex, we can give a sufficient condition for the frozen Jacobian algebra to be bimodule internally $3$-Calabi--Yau \cite[Defn.~2.4]{presslandinternally} with respect to the idempotent determined by the vertices of $F$. This condition will suffice as a definition for the purposes of the present paper.

\begin{thm}[{\cite[Thm.~5.7]{presslandinternally}}]
\label{frjaci3cy}
If $A$ is a frozen Jacobian algebra such that
\begin{equation}
\label{eq:bimod-res}
\begin{tikzcd}
0\arrow{r}&\res{A}\arrow{r}{\mu_0}&A\arrow{r}&0
\end{tikzcd}
\end{equation}
is exact, then $A$ is bimodule internally $3$-Calabi--Yau with respect to the idempotent $e=\sum_{v\in F_0}\idemp{v}$.
\end{thm}

\begin{rem}
\label{r:standard-res}
By standard results on presentations of algebras, as in Butler--King \cite{butlerminimal}, we need only check exactness at the two left-most non-zero terms of $\res{A}$ (in degrees $-3$ and $-2$), since the rest of the complex \eqref{eq:bimod-res} is the standard projective bimodule presentation of an algebra defined by a quiver with relations.
\end{rem}

\begin{prop}
\label{p:exactness-reduction}
For $A=A_D$ the dimer algebra of a connected Postnikov diagram, the complex \eqref{eq:bimod-res} is exact in degree $d$ if and only if the complex 
\begin{equation}
\label{eq:proj-res-simples}
\begin{tikzcd}
0\arrow{r}&\res{A}\tens_AS\arrow{r}{\mu_0\tens S}&S\arrow{r}&0
\end{tikzcd}
\end{equation}
is exact in degree $d$, if and only if the complex
\begin{equation}
\label{eq:proj-res-Sv}
\begin{tikzcd}
0\arrow{r}&\res{A}\tens_A\simp{v}\arrow{r}{\mu_0\tens\simp{v}}&\simp{v}\arrow{r}&0
\end{tikzcd}
\end{equation}
is exact in degree $d$ for any $v\in Q_0$, where $\simp{v}$ denotes the simple left $A$-module at $v$. 
\end{prop}
\begin{proof}
Exactness of \eqref{eq:bimod-res} and of \eqref{eq:proj-res-simples} are equivalent because of the existence of a grading as in Proposition~\ref{p:grading}, by an argument that is essentially due to Broomhead \cite[Prop.~7.5]{broomheaddimer}. The reader can find an explanation of how this argument extends to the case of frozen Jacobian algebras in \cite[\S4]{presslandcategorification}. The second equivalence is simply the observation that the complex \eqref{eq:proj-res-simples} is the direct sum of the complexes \eqref{eq:proj-res-Sv}.
\end{proof}

It will be useful to introduce some notation for elements of terms of $\res{A}\tens_A\simp{v}$, following \cite[\S4]{presslandcategorification}. First note that we have isomorphisms
\begin{align}
\label{eq:resA-isos}
\begin{split}
A\tens\KK Q_1\tens A\tens_A\simp{v}&\iso\bigdsum_{b\in\Tail{v}}Ae_{\head{b}},\\
A\tens\KK Q_2^\mut\tens A\tens_A\simp{v}&\iso\bigdsum_{a\in\mHead{v}}Ae_{\tail{a}},\\
A\tens\KK Q_3^\mut\tens A\tens_A\simp{v}&\iso{\begin{cases}A\idemp{v},&v\in Q_0^\mut,\\0,&v\in F_0.\end{cases}}
\end{split}
\end{align}

In the first two cases, the right-hand sides are of the form $\bigoplus_{a\in S}A\idemp{\delta a}$, where $S$ is a set of arrows, and $\delta\colon S\to Q_0$. The map $A\idemp{\delta a}\to\bigoplus_{a\in S}A\idemp{\delta a}$ including the domain as the summand indexed by $a$ will be denoted by $x\mapsto x\otimes[a]$; this helps us to distinguish these various inclusions when $\delta$ is not injective. As a consequence, a general element of the direct sum is
\begin{equation}
\label{eq:dsum-notation}
x=\sum_{a\in S}x_a\otimes[a]
\end{equation}
with $x_a\in A\idemp{\delta a}$.

The most complicated map in the complex $\res{A}$ is $\mu_2$, so we will spell out $\mu_2\tens_A\simp{v}$ explicitly. Using the isomorphisms from \eqref{eq:resA-isos}, and our notation for elements of the direct sums, we have
\[(\mu_2\tens_A\simp{v})(x)=\sum_{b\in\Tail{v}}\bigg(\sum_{a\in\mHead{v}}x_a\rightder{b}{\der{a}{W}}\bigg)\tens[b],\]
where $\rightder{b}$, called the \emph{right derivative} with respect to $b$, is defined on paths by
\begin{equation}
\label{eq:rightder}
\rightder{b}(a_k\dotsm a_1)=\begin{cases}a_k\dotsm a_2,&a_1=b,\\0,&a_1\ne b\end{cases}
\end{equation}
and extended linearly and continuously. Similarly, there is an \emph{left derivative}, defined on paths by
\begin{equation}
\label{eq:leftder}
\leftder{b}(a_k\dotsm a_1)=\begin{cases}a_{k-1}\dotsm a_1,&a_k=b,\\0,&a_k\ne b.\end{cases}
\end{equation}
Given a Postnikov diagram $D$ and two arrows $a$ and $b$ of $Q_D$, we may observe that
\begin{equation}
\label{eq:der-interchange}
\rightder{b}{\der{a}{W}}=\leftder{a}{\der{b}{W}}.
\end{equation}
Indeed, there are at most two fundamental cycles in $Q_D$ containing $a$, and so the part of $W$ consisting of terms containing both $a$ and $b$ is of the form $apb-aqb$ for some $p$ and $q$ which are either paths, not containing the arrows $a$ or $b$, or are zero. Thus one can directly calculate each side of \eqref{eq:der-interchange}, which both result in $p-q$.

We are now ready to prove our main theorem.

\begin{thm}
\label{t:bi3cy}
For $A=A_D$ the dimer algebra of a connected Postnikov diagram $D$, the complex \eqref{eq:bimod-res} is exact, and hence $A$ is bimodule internally $3$-Calabi--Yau with respect to the idempotent determined by the boundary vertices.
\end{thm}

\begin{proof}
As already noted in Remark~\ref{r:standard-res}, the complex \eqref{eq:bimod-res} can only fail to be exact in degrees $-2$ and $-3$, and so, by Proposition~\ref{p:exactness-reduction}, we need only check that the complex \eqref{eq:proj-res-Sv} is exact at the terms $A\tens\KK Q_3^\mut\tens\simp{v}$ and $A\tens\KK Q_2^\mut\tens\simp{v}$ for each $v\in Q_0$.

We first deal with exactness at $A\tens\KK Q_3^\mut\tens\simp{v}$. This term is zero unless $v\in Q_0^\mut$, so we may additionally assume this. By the third isomorphism in \eqref{eq:resA-isos}, we then have $A\tens\KK Q_3^\mut\tens\simp{v}\cong A\idemp{v}$, so let $x\in A\idemp{v}$. We calculate
\[(\mu_3\otimes\simp{v})(x)=\sum_{a\in\mHead{v}}xa\otimes[a],\]
using the notation \eqref{eq:dsum-notation} for elements of $A\tens\KK Q_2^\mut\tens\simp{v}\cong\bigdsum_{a\in\mHead{v}}A\idemp{\tail{a}}$. Thus $x\in\ker(\mu_3\otimes\simp{v})$ if and only if $xa=0$ for all $a\in\mHead{v}$. Note that since $v$ is mutable, $\mHead{v}=\Head{v}\neq\varnothing$.

Let $w\in Q_0$, so $\idemp{w}x\in\idemp{w}A\idemp{v}$. Since $A$ is thin by Proposition~\ref{p:thin}, we have $\idemp{w}A\idemp{v}=Z\minpath{v}{w}$, so we may write $\idemp{w}x=\sum_{n\in\NN}z_nt^n\cdot\minpath{v}{w}$ for some sequence of scalars $z_n\in\KK$. Using thinness again, for any $a\in\mHead{v}$ there is some $\delta_a\in\NN$ such that $\minpath{v}{w}a=t^{\delta_a}\minpath{,\tail{a}}{w}$. It then follows that
\[\idemp{w}xa=\sum_{n\in\NN}z_nt^n\cdot\minpath{v}{w}a=\sum_{n\in\NN}z_nt^{n+\delta_a}\cdot\minpath{,\tail{a}}{w}.\]
Thus if $xa=0$, we have $z_n=0$ for all $n$, and so $\idemp{w}x=0$. Since $w$ was chosen arbitrarily, we conclude that $x=0$, and so the kernel of $\mu_3\otimes\simp{v}$ is trivial. This establishes exactness of \eqref{eq:proj-res-Sv} at the term $A\tens\KK Q_3^\mut\tens\simp{v}$.

We now move to the term $A\tens\KK Q_2^\mut\tens\simp{v}$. Using the relevant isomorphisms in \eqref{eq:resA-isos} and the notation of \eqref{eq:dsum-notation}, a general element of this term is of the form $\phi=\sum_{a\in\mHead{v}}x_a\tens[a]$ for $x\in A\idemp{\tail{a}}$,
and its image under $\mu_2\tens\simp{v}$ is
\[(\mu_2\tens\simp{v})(\phi)=\sum_{b\in\Tail{v}}\Bigg(\sum_{a\in\mHead{v}}x_a\rightder{b}{\der{a}{W}}\Bigg)\tens[b]=\sum_{b\in\Tail{v}}\Bigg(\sum_{a\in\mHead{v}}x_a\leftder{a}{\der{b}{W}}\Bigg)\tens[b],\]
where the second equality follows from \eqref{eq:der-interchange}. The reader is warned that, since $b\in\Tail{v}$ may be frozen, the derivative $\der{b}{W}$ is not necessarily zero in $A$.

Now assume that $(\mu_2\tens\simp{v})(\phi)=0$, or equivalently by the above calculation that
\begin{equation}
\label{eq:mu2image}
\sum_{a\in\mHead{v}}x_a\leftder{a}{\der{b}{W}}=0
\end{equation}
for all $b\in\Tail{v}$. Picking $w\in Q_0$ and multiplying by $\idemp{w}$ on the left, we obtain elements $\idemp{w}x_a\in\idemp{w}A\idemp{\tail{a}}$ for each $a\in\mHead{v}$.

As in the first part of the proof, we will now use thinness of $A$ to write the elements $\idemp{w}x_a$ as the product of a power series in $t$ with $\minpath{,\tail{a}}{w}$, but in order to keep the notation clean in the subsequent argument, we will do this in a slightly unusual way. First note that, having fixed the vertex $w$, any path $p$ determines $\delta_p\in\NN$ with the property that $\minpath{,\head{p}}{w}p=t^{\delta_p}\minpath{,\tail{p}}{w}$. Since this equality is formulated in the algebra $A$, the quantity $\delta_p$ depends only on the class of $p$ in the algebra. 

Using this notation, we can write
\begin{equation}
\label{eq:t-expansion}
\idemp{w}x_a=t^{\delta_a}\sum_{n\in\ZZ}z_n^{w,a}t^n\cdot\minpath{,\tail{a}}{w}
\end{equation}
for some scalars $z_n^{w,a}\in\KK$. Since the right-hand side of this expression may only involve non-negative powers of $t$, we have $z_n^{w,a}=0$ whenever $n<-\delta_a$.

The key step in the argument is to show, under our assumption that $(\mu_2\tens\simp{v})(\phi)=0$, that the coefficient $z_n^{w,a}$ depends only on $n$ and $w$, and not on $a\in\mHead{v}$. We consider the configuration
\begin{equation}
\label{eq:arrow-config}
\begin{tikzcd}[column sep=40pt,row sep=40pt]
&v\arrow[quivarrow]{d}{b}\\
\bullet\arrow[quivarrow,dotted]{dr}[swap]{\idemp{w}x_-}\arrow[quivarrow]{ur}{a_-}&\bullet\arrow[quivarrow,dashed]{l}[swap]{q}\arrow[quivarrow,dashed]{r}{p}&\bullet\arrow{ul}[swap]{a_+}\arrow[quivarrow,dotted]{dl}{\idemp{w}x_+}\\
&w
\end{tikzcd}
\end{equation}
Here solid arrows represent arrows in $Q_1$, dashed arrows represent paths in $Q$, and dotted arrows represent linear combinations of paths. Our assumption is that $a_+$ and $a_-$ are unfrozen arrows, with head at $v$ as depicted, which appear in the two fundamental cycles $a_+pb$ and $a_-qb$ containing the arrow $b$, also unfrozen, with tail at $v$. As indicated by the notation in \eqref{eq:arrow-config}, we will usually omit the letter $a$ when $a_+$ or $a_-$ appear in subscripts or superscripts in the following argument---for example this means that $x_+:=x_{a_+}$. Our aim is to show that $z_n^{w,+}=z_n^{w,-}$ for all $n\in\ZZ$; it will then follow from the way that the quiver $Q$ is embedded in the disc that $z_n^{w,a}$ is independent of $a\in\mHead{v}$, since if $a,a'\in\mHead{v}$, we can find sequences $a_0,\dotsc,a_k$ and $b_1\cdots b_k$ such that $a_0=a$, $a_k=a'$, and for all $1\leq i\leq k$ the triple of arrows $(a_{i-1},b_i,a_i)$ has the same configuration as $(a_-,b,a_+)$ in \eqref{eq:arrow-config} (or its mirror image), so that $z_n^{w,a_{i-1}}=z_n^{w,a_i}$.

So consider the configuration shown in \eqref{eq:arrow-config}. We see from the figure that $\der{b}{W}=a_+p-a_-q$. While it can happen that $a_+=a_-$ (if and only if $v$ is a bivalent vertex of $Q$), in this case there is nothing to prove, so we may assume $a_+\ne a_-$. Then for $a\in\mHead{v}$, we have
\[\leftder{a}{\der{b}{W}}=\begin{cases}p,&a=a_+,\\-q,&a=a_-,\\0,&\text{otherwise}.\end{cases}\]
Thus
\[\sum_{a\in\mHead{v}}x_a\leftder{a}{\der{b}{W}}=x_+p-x_-q.\]
Since $\phi=\sum_{a\in\mHead{v}}x_a\tens[a]$ is in the kernel of $\mu_2\tens\simp{v}$, this quantity is zero by \eqref{eq:mu2image}, and so $x_+p=x_-q$. It follows that
\begin{align*}
\idemp{w}x_+p&=t^{\delta_+}\sum_{n\in\ZZ}z_n^{w,+}t^n\cdot\minpath{,\tail{a_+}}{w}p=t^{\delta_++\delta_p}\sum_{n\in\ZZ}z_n^{w,+}t^n\cdot\minpath{,\head{b}}{w},\\
\idemp{w}x_-q&=t^{\delta_-}\sum_{n\in\ZZ}z_n^{w,-}t^n\cdot\minpath{,\tail{a_-}}{w}q=t^{\delta_-+\delta_q}\sum_{n\in\ZZ}z_n^{w,-}t^n\cdot\minpath{,\head{b}}{w},
\end{align*}
must also be equal. Since $a_+p=a_-q$ in $A$ (because of the relation $\der{b}{W}$), we must have $\delta_++\delta_p=\delta_-+\delta_q$, and so we can conclude, for any $n\in\ZZ$, that $z_n^{w,+}=z_n^{w,-}$, and further that $z_n^{w,a}$ is independent of $a\in\mHead{v}$, as required. From now on, we abbreviate $z_n^w:=z_n^{w,a}$, and for any $a\in\mHead{v}$ we have
\begin{equation}
\label{eq:t-expansion2}
\idemp{w}x_a=t^{\delta_a}\sum_{n\in\ZZ}z_n^wt^n\cdot\minpath{,\tail{a}}{w}.
\end{equation}

We now return to our main purpose, to show that the complex \eqref{eq:bimod-res} is exact. We recall that we are assuming that $\phi=\sum_{a\in\mHead{v}}x_a\tens[a]$ satisfies $(\mu_2\tens\simp{v})(\phi)=0$, and wish to show that $\phi$ is in the image of $\mu_3\tens\simp{v}$.

First we deal with the case that $v$ is a frozen vertex, in which case $\mu_3\tens\simp{v}=0$, and we wish to conclude that $x_a=0$ for all $a\in\mHead{v}$. This is equivalent to showing that the complex numbers $z_n^w$ appearing in \eqref{eq:t-expansion2} are zero for all $n$ and all $w$.

It follows from the construction of $(Q_D,F_D)$ that, since $v$ is frozen, it is incident with a frozen arrow. If $\mHead{v}=\varnothing$ then we have nothing to show, and otherwise there is an arrow $a\in\mHead{v}$ as in one of the following configurations
\begin{equation}
\label{f:boundary-config}
\begin{tikzcd}[column sep=40pt,row sep=40pt]
\bullet\arrow[quivarrow]{dr}{a}&&\bullet\arrow[quivarrow,dashed,swap]{d}{q}\arrow[quivarrow,dashed]{r}{p}&\bullet\arrow[quivarrow]{d}{a}\\
\diamond\arrow[quivarrow,dashed]{u}{p}&v\arrow[frozarrow]{l}{b}&\diamond\arrow[frozarrow,swap]{r}{c}&v\arrow[quivarrow,swap]{ul}{b}
\end{tikzcd}
\end{equation}
in which $apb$ and $cqb$ are fundamental cycles. As earlier in the paper, the bold arrows are frozen, and the dashed arrows represent paths. The vertices indicated by diamonds are frozen, like the vertex $v$, whereas the others may be mutable or frozen.

In the two configurations in \eqref{f:boundary-config} we have $\der{b}{W}=\pm ap$ and $\der{b}{W}=\pm(ap-cq)$ respectively. Thus in either case we calculate for $a'\in\mHead{v}$ that
\[\leftder{a'}{\der{b}{W}}=\begin{cases}\pm p,&a'=a,\\0,&\text{otherwise}.\end{cases}\]
Since $(\mu_2\tens\simp{v})(\phi)=0$, we deduce from \eqref{eq:mu2image} that
\[x_ap=\pm\sum_{a'\in\mHead{v}}x_a\leftder{a'}\der{b}{W}=0.\]
Multiplying on the left by $\idemp{w}$ and using that $A$ is thin, we see that
\[0=\idemp{w}x_ap=t^{\delta_a}\sum_{n\in\ZZ}z_n^wt^n\cdot\minpath{,\tail{a}}{w}p=t^{\delta_a+\delta_p}\sum_{n\in\ZZ}z_n^wt^n\cdot{\minpath{,\head{b}}{w}},\]
and so $z_n^w=0$ for all $n$ and $w$, as required.

Now we treat the case that $v$ is a mutable vertex; in this case $\mHead{v}=\Head{v}$ and we use the latter to keep the notation simpler. In this case, we want to construct $y\in A\idemp{v}$ such that
\[(\mu_3\tens\simp{v})(y)=\sum_{a\in\Head{v}}ya\tens[a]=\phi.\]
Using the sequence of complex numbers $z_n^w$ appearing in \eqref{eq:t-expansion2}, we write
\[y_w=\sum_{n\in\ZZ}z_n^wt^n\cdot\minpath{v}{w}.\]
For this to be a well-defined element of $\idemp{w}A\idemp{v}$, we need $z_n^w=0$ for all $n<0$. Choosing any $a\in\Head{v}$, we have that $z_n^w=z_n^{w,a}$ is zero whenever $n<-\delta_a$, as in \eqref{eq:t-expansion}. But by Lemma~\ref{l:nb-lemma}, there is some $a\in\Head{v}$ such that $\delta_a=0$, and so $y_w\in\idemp{w}A\idemp{v}$ as required.

Now, letting $y=\sum_{w\in Q_0}y_w\in A\idemp{v}\cong A\tens\KK Q_3^\mut\tens\simp{v}$, we can calculate
\begin{align*}
(\mu_3\tens\simp{v})(y)=\sum_{w\in Q_0}\sum_{a\in\Head{v}}y_wa\tens[a]&=\sum_{w\in Q_0}\sum_{a\in\Head{v}}\Bigg(\sum_{n\in\NN}z_n^wt^n\cdot\minpath{v}{w}a\Bigg)\tens[a]\\
&=\sum_{w\in Q_0}\sum_{a\in\Head{v}}\Bigg(\sum_{n\in\NN}z_n^wt^{n+\delta_a}\cdot\minpath{,\tail{a}}{w}\Bigg)\tens[a]\\
&=\sum_{w\in Q_0}\sum_{a\in\Head{v}}\idemp{w}x_a\tens[a]=\sum_{a\in\Head{v}}x_a\tens[a]=\phi,
\end{align*}
and so we conclude that $\phi\in\im(\mu_3\tens\simp{v})$, as required.
\end{proof}

\section{Categorification}
\label{s:categorification}

Given a suitable algebra $A$ which is bimodule internally $3$-Calabi--Yau with respect to an idempotent $e$, one can construct from $A$ and $e$ a Frobenius category having many desirable properties from the point of view of categorifying cluster algebras. This is done via the main result of \cite{presslandinternally} (specialised to the case that the Calabi--Yau dimension is $3$), which we will recall after collecting the necessary algebraic definitions. Applied to the dimer algebra $A_D$ of a connected Postnikov diagram $D$, this will allow us to show (Theorem~\ref{t:categorification}) that the category $\GP(B_D)$ appearing in Theorem~\ref{t:mainthm} has many of the properties required of an additive categorification of the cluster algebra $\clustalg{D}$. Further such properties will be proved in Section~\ref{s:cluster-character}.

\begin{defn}
\label{d:cat-defs}
Let $\frobcat$ be an exact $\KK$-linear category, as in \cite{buehlerexact}. We do not recall the full definition of this structure here, but simply note that any full extension-closed subcategory $\frobcat$ of an abelian category $\abcat$ has a natural structure of an exact category, by taking the class of short exact sequences to be the short exact sequences of $\abcat$ with terms in $\frobcat$.

We say $\frobcat$ is a \emph{Frobenius category} if it has enough projective objects and enough injective objects, and these two classes of objects coincide. The \emph{stable category} $\stab{\frobcat}$ is obtained from $\frobcat$ by factoring out the ideal of morphisms which factor over a projective (equivalently injective) object, and has a natural triangulated structure \cite[\S I.2]{happeltriangulated}. We say that $\stab{\frobcat}$ is \emph{$2$-Calabi--Yau}, and $\frobcat$ is \emph{stably $2$-Calabi--Yau}, if $\stab{\frobcat}$ is Hom-finite and there is a functorial isomorphism
\[\Hom_{\stab{\frobcat}}(X,Y)=\Kdual\Hom_{\stab{\frobcat}}(Y,\Sigma^2 X)\]
for any two objects $X$ and $Y$ of $\stab{\frobcat}$, where $\Sigma$ is the shift functor in the triangulated structure of $\stab{\frobcat}$ (computed by taking the cokernel of an injective envelope) and $\Kdual$ is duality over the ground field $\KK$. 

A full subcategory $\ctcat\subset\frobcat$ is called \emph{cluster-tilting} if it is functorially finite and
\[\ctcat=\{X\in\frobcat:\Ext^1_\frobcat(X,T)=0\ \forall\ T\in \ctcat\}=\{X\in\frobcat:\Ext^1_\frobcat(T,X)=0\ \forall\ T\in\ctcat\}.\]
Note that cluster-tilting subcategories are necessarily additively closed, i.e.\ closed under direct sums and summands. We can make the same definition in $\stab{\frobcat}$, using the usual definition $\Ext^1_{\stab{\frobcat}}(M,N)=\Hom_{\stab{\frobcat}}(M,\Sigma N)$ for triangulated categories, and observe that a subcategory $\ctcat\subset\frobcat$ containing all projective objects is cluster-tilting if and only if its image in the stable category $\stab{\frobcat}$ (i.e.\ the full subcategory of $\stab{\frobcat}$ on the same set of objects as $\ctcat$) is cluster-tilting. When $\frobcat$ is stably $2$-Calabi--Yau, the second equality in the definition of a cluster-tilting subcategory is automatic. An object $T\in\frobcat$ is a \emph{cluster-tilting object} if its additive closure $\add{T}$ is a cluster-tilting subcategory.

An algebra $B$ is called \emph{Iwanaga--Gorenstein} if it is Noetherian and has finite injective dimension both as a left and as a right module over itself. In this case these two injective dimensions coincide \cite{iwanagarings2}, and we call this number the \emph{Gorenstein dimension} of $B$. For an Iwanaga--Gorenstein algebra $B$, we define the category of \emph{Gorenstein projective} $B$-modules to be the full subcategory
\[\GP(B)=\set{X\in\fgmod{B}:\text{$\Ext^i_B(X,B)=0$ for all $i>0$}}\]
of the abelian category of finitely generated $B$-modules. This is a Frobenius category; see for example \cite[\S4]{buchweitzmaximal}, where this category is denoted by $\MCM(B)$ and its objects called maximal Cohen--Macaulay modules.
\end{defn}

\begin{rem}
The reader is warned that, when $B$ is free and finitely generated over $Z=\powser{\KK}{t}$, as will be the case in our examples, one can define the (maximal) Cohen--Macaulay $B$-modules to be those $B$-modules which are free and finitely generated over $Z$, as in \cite[\S3]{jensencategorification}. While a module for the algebra $C$ defining the Grassmannian cluster category is Gorenstein projective if and only if it is Cohen--Macaulay \cite[Cor.~3.7]{jensencategorification}, this will not be the case for $B$ in general (see Example~\ref{eg:CM-not-GP}). To avoid having two non-equivalent definitions of Cohen--Macaulay $B$-modules, we prefer the terminology `Gorenstein projective' for the homological condition. The reader is however further warned that Gorenstein projective modules are defined over an arbitrary ring (as in \cite[Defn.~2.1]{holmgorenstein} for example), but in a way that need not be equivalent to Definition~\ref{d:cat-defs} if the ring is not Iwanaga--Gorenstein.
\end{rem}

\begin{thm}[{\cite[Thm.~4.1, Thm.~4.10]{presslandinternally}}]
\label{t:icytofrobcat}
Let $A$ be an algebra, and $e\in A$ an idempotent. If $A$ is Noetherian, $\stab{A}=A/\Span{e}$ is finite-dimensional, and $A$ is bimodule internally $3$-Calabi--Yau with respect to $e$, then
\begin{enumerate}
\item $B=eAe$ is Iwanaga--Gorenstein with Gorenstein dimension at most $3$,
\item $eA$ is a cluster-tilting object in the Frobenius category $\GP(B)$,
\item the stable category $\stabGP(B)$ is a $2$-Calabi--Yau triangulated category, and
\item the natural maps $A\to\Endalg{B}{eA}$ and $\stab{A}\to\stabEndalg{B}{eA}$ are isomorphisms.
\end{enumerate}
\end{thm}

We already saw in Theorem~\ref{t:bi3cy} that the dimer algebra $A_D$ of a connected Postnikov diagram $D$ is bimodule internally $3$-Calabi--Yau with respect to its boundary idempotent. Checking the other conditions needed to apply Theorem~\ref{t:icytofrobcat} to $A_D$ is comparatively straightforward, although we note that we once again rely heavily on Proposition~\ref{p:thin}.

\begin{prop}
\label{p:smallness}
Let $D$ be a connected Postnikov diagram with dimer algebra $A_D$, and let $e$ be the boundary idempotent of this algebra. Then $A_D$ is Noetherian and $A_D/\Span{e}$ is finite-dimensional.
\end{prop}
\begin{proof}
Since $A_D$ is thin, as in Proposition~\ref{p:thin}, it is finitely generated as a module over the commutative Noetherian ring $\powser{\KK}{t}$, and so is Noetherian.

By thinness again, $A_D$ has a basis $\{t^n\minpath{v}{w}:v,w\in Q_0,n\in\NN\}$. To see that $A_D/\Span{e}$ is finite-dimensional, we show that all but finitely many of these basis vectors are zero in the quotient. Picking any $v,w\in Q_0$, let $c$ be any cycle at $w$ passing through a boundary vertex, which exists since $Q_D$ is strongly connected as in Proposition~\ref{p:strongly-connected}. Then since $A_D$ is thin, we have $c\minpath{v}{w}=t^N\minpath{v}{w}$ for some $N\in\NN$. By construction, this element is zero in the quotient $A_D/\Span{e}$, and thus $t^n\minpath{v}{w}=0$ in this quotient algebra for all $n\geq N$. Running over all pairs $v,w\in Q_0$ completes the proof.
\end{proof}

Recall that we write $\clustalg{D}$ for the cluster algebra with frozen variables determined by the ice quiver $(Q_D,F_D)$ of a Postnikov diagram $D$. Note that arrows between frozen vertices play no role in the construction of this cluster algebra. Since $Q_D$ may contain $2$-cycles, we are committing a small abuse of notation here; really $\clustalg{D}$ is the cluster algebra determined by the quiver $Q'$ of the unique reduced Postnikov diagram obtained from $D$ by untwisting moves, as in Figure~\ref{f:untwist}. Equivalently, $Q'$ is obtained from $Q_D$ by removing a maximal collection of $2$-cycles incident with unfrozen vertices. Using the description via untwisting and Proposition~\ref{p:untwisting}, we see that $Q'$ has a potential $W'$ such that $\frjac{Q'}{F_D}{W'}\cong A_D$ is the dimer algebra of the original diagram $D$.

By Theorem~\ref{t:bi3cy} and Proposition~\ref{p:smallness}, we may apply Theorem~\ref{t:icytofrobcat} to the dimer algebra $A_D$ of a connected Postnikov diagram, to obtain the following.

\begin{thm}
\label{t:categorification}
Let $D$ be a connected Postnikov diagram, with dimer algebra $A_D$. Let $e$ be the boundary idempotent of this algebra, and write $B_D=eA_De$. Then
\begin{enumerate}
\item  $B_D$ is Iwanaga--Gorenstein of Gorenstein dimension at most $3$,
\item $\GP(B_D)$ is a stably $2$-Calabi--Yau Frobenius category,
\item $T_D=eA_D\in\GP(B_D)$ is cluster-tilting, and
\item $\Endalg{B_D}{T_D}\iso A_D$.
\end{enumerate}
In particular, it follows from (4) that the Gabriel quiver of $\Endalg{B_D}{T_D}$ is, up to arrows between frozen vertices, the quiver of a seed in the cluster algebra $\clustalg{D}$.
\end{thm}

We note that if the full subquiver of $Q_D$ on the mutable vertices is, after removing $2$-cycles, mutation-equivalent to an acyclic quiver $\stab{Q}$, it follows from a result of Keller and Reiten \cite[Thm.~2.1]{kelleracyclic} that the stable category $\stabGP(B_D)$ of our categorification is equivalent to the cluster category $\clustcat{\stab{Q}}$ \cite{buancluster}. This mutation-acyclicity assumption does not hold for most Postnikov diagrams, however. We conjecture that $\stabGP(B_D)$ is always equivalent to Amiot's cluster category \cite{amiotcluster} for the quiver with potential obtained from $(Q_D,F_D,W_D)$ by deleting the frozen vertices and all incident arrows from $Q_D$ and removing terms passing through these vertices from $W_D$.

If $D$ is a $(k,n)$-diagram, as in Definition~\ref{d:postnikov-diag}, it follows from \cite[Cor.~10.4]{baurdimer} that the boundary algebra $B_D=eA_De$ does not depend on the choice of $D$ within this class, up to isomorphism, and the category $\GP(B_D)$ is equivalent to Jensen--King--Su's categorification of the cluster algebra structure on the whole Grassmannian \cite{jensencategorification}.

\begin{figure}[h]
\begin{tikzpicture}
\draw[dashed] plot[smooth]
coordinates {(150:2.3) (135:1) (90:0.6) (45:1) (30:2.3)}
[postaction=decorate, decoration={markings,
 mark= at position 0.1 with \strarrow,
 mark= at position 0.5 with \strarrow,
 mark= at position 0.93 with \strarrow}];
 
\draw[dashed] plot[smooth]
coordinates {(120:2.3) (135:1) (180:0.6) (225:1) (240:2.3)}
[postaction=decorate, decoration={markings,
 mark= at position 0.1 with \strarrow,
 mark= at position 0.53 with \strarrow,
 mark= at position 0.93 with \strarrow}];
 
\draw[dashed] plot[smooth]
coordinates {(330:2.3) (315:1) (270:0.6) (225:1) (210:2.3)}
[postaction=decorate, decoration={markings,
 mark= at position 0.1 with \strarrow,
 mark= at position 0.5 with \strarrow,
 mark= at position 0.93 with \strarrow}];
 
\draw[dashed] plot[smooth]
 coordinates {(300:2.3) (315:1) (0:0.6) (45:1) (60:2.3)}
[postaction=decorate, decoration={markings,
 mark= at position 0.1 with \strarrow,
 mark= at position 0.53 with \strarrow,
 mark= at position 0.93 with \strarrow}];

\node at (0:0) (R0) {$v_0$};
\foreach \n in {1,2,3,4}
{\node at (-45+90*\n:2.3) (R\n) {$v_\n$};}
\draw[quivarrow]
(R0) -- (R1);
\draw[quivarrow]
(R2) -- (R0);
\draw[quivarrow]
(R0) -- (R3);
\draw[quivarrow]
(R4) -- (R0);
\end{tikzpicture}
\hspace{5em}
\begin{tikzpicture}
\draw[dashed] plot[smooth]
coordinates {(150:2.3) (180:1.3) (270:0.6) (0:1.3) (30:2.3)}
[postaction=decorate, decoration={markings,
 mark= at position 0.1 with \strarrow,
 mark= at position 0.33 with \strarrow,
 mark= at position 0.51 with \strarrow,
 mark= at position 0.7 with \strarrow,
 mark= at position 0.9 with \strarrow}];
 
\draw[dashed] plot[smooth]
coordinates {(120:2.3) (90:1.3) (0:0.6) (270:1.3) (240:2.3)}
[postaction=decorate, decoration={markings,
 mark= at position 0.1 with \strarrow,
 mark= at position 0.33 with \strarrow,
 mark= at position 0.51 with \strarrow,
 mark= at position 0.7 with \strarrow,
 mark= at position 0.9 with \strarrow}];
 
\draw[dashed] plot[smooth]
coordinates {(330:2.3) (0:1.3) (90:0.6) (180:1.3) (210:2.3)}
[postaction=decorate, decoration={markings,
 mark= at position 0.1 with \strarrow,
 mark= at position 0.33 with \strarrow,
 mark= at position 0.51 with \strarrow,
 mark= at position 0.7 with \strarrow,
 mark= at position 0.9 with \strarrow}];
 
\draw[dashed] plot[smooth]
 coordinates {(300:2.3) (270:1.3) (180:0.6) (90:1.3) (60:2.3)}
[postaction=decorate, decoration={markings,
 mark= at position 0.1 with \strarrow,
 mark= at position 0.33 with \strarrow,
 mark= at position 0.51 with \strarrow,
 mark= at position 0.7 with \strarrow,
 mark= at position 0.9 with \strarrow}];
 
\node at (0:0) (R0) {$v_0$};
\foreach \n in {1,2,3,4}
{\node at (-45+90*\n:1.85) (R\n) {$v_\n$};}
\draw[quivarrow]
(R1) -- (R0);
\draw[quivarrow]
(R0) -- (R2);
\draw[quivarrow]
(R3) -- (R0);
\draw[quivarrow]
(R0) -- (R4);
\draw[quivarrow]
(R2) -- (R1);
\draw[quivarrow]
(R4) -- (R1);
\draw[quivarrow]
(R4) -- (R3);
\draw[quivarrow]
(R2) -- (R3);
\end{tikzpicture}
\caption{A \emph{geometric exchange} of a Postnikov diagrams transforms the local configuration of a quadrilateral alternating region, shown on the left, to that shown on the right. The effect on the quiver with potential is a mutation \cite[\S5]{derksenquivers1} at the vertex corresponding to the quadrilateral alternating region (cf.\ \cite{vitoriamutations}).}
\label{f:geom-exch}
\end{figure}
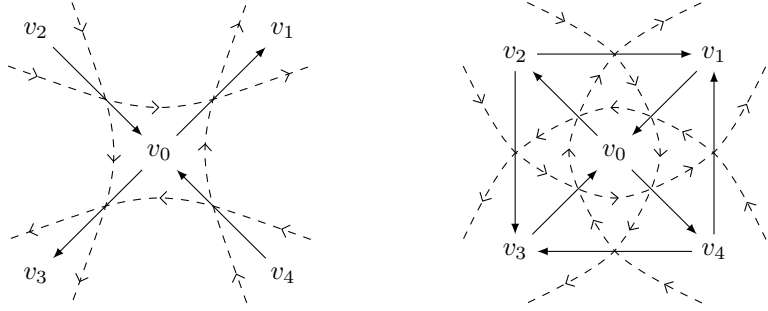

In general, we can consider two Postnikov diagrams $D$ and $D'$ related by geometric exchange as shown in Figure~\ref{f:geom-exch}. Combining \cite[Thm.~13.4]{postnikovtotal} and \cite[Thm.~17.1]{postnikovtotal}, we see that these diagrams determine the same positroid, and hence the same positroid varieties. Moreover, the quivers $Q_D$ and $Q_{D'}$ are related by a mutation, and so the cluster algebras $\clustalg{D}$ and $\clustalg{Q_{D'}}$ are isomorphic. Thus we would like both $D$ and $D'$ to give rise to the same category of Gorenstein projective modules in Theorem~\ref{t:categorification}. This is indeed the case, as follows.

\begin{prop}
\label{p:geom-exch}
Let $D$ and $D'$ be Postnikov diagrams related by geometric exchange. Then the boundary algebras $B_D$ and $B_{D'}$ are isomorphic, and hence the categories $\GP(B_D)$ and $\GP(B_{D'})$ are equivalent.
\end{prop}
\begin{proof}
This can be proved exactly as in \cite[Prop.~12.2]{baurdimer}; while this proposition is stated only for $(k,n)$-diagrams, the proof does not use the extra assumption on the strand permutation.
\end{proof}

\begin{cor}
\label{c:dec-perm}
Given a Postnikov diagram $D$, the algebra $B_D$ is determined up to isomorphism by the positroid $\posit_D$.
\end{cor}
\begin{proof}
Again combining \cite[Thm.~13.4]{postnikovtotal} and \cite[Thm.~17.1]{postnikovtotal}, any two diagrams determining the same positroid are connected by a sequence of geometric exchanges, so the result follows from Proposition~\ref{p:geom-exch}.
\end{proof}

One consequence of the general results in Theorem~\ref{t:icytofrobcat} on bimodule internally $3$-Calabi--Yau algebras is that the Gorenstein dimension of $B_D$ is at most $3$ for any connected Postnikov diagram $D$. In fact, this Gorenstein dimension never exceeds $2$ by \cite[Cor.~10.5]{canakciperfect}. When $D$ is a $(k,n)$-diagram, $B_D$ even has Gorenstein dimension $1$, as in the proof of \cite[Cor.~3.7]{jensencategorification}. Since $1$ is also the Krull dimension of $Z=\powser{\KK}{t}$, in this case the category $\GP(B_D)$ coincides with the category of Cohen--Macaulay $B_D$-modules, meaning those $B_D$-modules free and finitely generated over $Z$. However, the Gorenstein dimension of $B_D$ can be equal to $2$, and the inclusion $\GP(B_D)\to\CM(B_D)$ can be proper, as the following example shows.

\begin{eg}
\label{eg:CM-not-GP}
Consider the Postnikov diagram $D$ shown in Figure~\ref{f:Gdim2} and its dimer algebra $A_D$, the quiver of which is shown in the same figure. Since every alternating region of $D$ is on the boundary (equivalently, every vertex of the quiver is frozen), we have $B_D=A_D$, which has finite global dimension by Theorem~\ref{t:bi3cy}. In fact, since the left-most term of $\res{A_D}$ is zero in this case, again because all vertices are frozen, the global dimension of $B_D$ is $2$.

\begin{figure}[h]
\begin{minipage}[t]{0.4\textwidth}
\centering
\begin{tikzpicture}[scale=2,baseline=(bb.base)]

\path (0,0) node (bb) {};

\draw [boundary] (0,0) circle(1.0);

\draw [strand] plot[smooth]
coordinates {(45:1) (65:0.65) (0:0) (235:0.65) (225:1)}
[postaction=decorate, decoration={markings,
 mark= at position 0.3 with \strarrow,
  mark= at position 0.75 with \strarrow}];
  
\draw [strand] plot[smooth]
coordinates {(135:1) (115:0.65) (0:0) (305:0.65) (315:1)}
[postaction=decorate, decoration={markings,
 mark= at position 0.3 with \strarrow,
  mark= at position 0.75 with \strarrow}];
  
\draw [strand] plot[smooth]
coordinates {(225:1) (180:0.5) (135:1)}
[postaction=decorate, decoration={markings,
  mark= at position 0.5 with \strarrow}];
  
\draw [strand] plot[smooth]
coordinates {(315:1) (0:0.5) (45:1)}
[postaction=decorate, decoration={markings,
  mark= at position 0.5 with \strarrow}];
\end{tikzpicture}
\end{minipage}
\begin{minipage}[t]{0.4\textwidth}
\centering
\begin{tikzpicture}[scale=2,baseline=(bb.base)]

\path (0,0) node (bb) {};

\foreach \n in {1, 2, 3, 4}
{ \draw (90*\n:1) node (v\n) {$\diamond$}; }

\foreach \t/\h/\a in {4/1/-35, 2/1/35, 3/2/35, 3/4/-35}
{ \draw [frozarrow]  (v\t) edge [bend left=\a] (v\h); }
 
\draw [quivarrow]  (v1) edge (v3);
\end{tikzpicture}
\end{minipage}
\caption{A Postnikov diagram $D$ of type $(2,4)$ giving rise to a boundary algebra $B_D$ for which not all Cohen--Macaulay modules are Gorenstein projective. The right-hand figure shows the quiver $Q_D$, with frozen arrows marked in bold as usual, which is the Gabriel quiver of $A_D=B_D$.}
\label{f:Gdim2}
\end{figure}
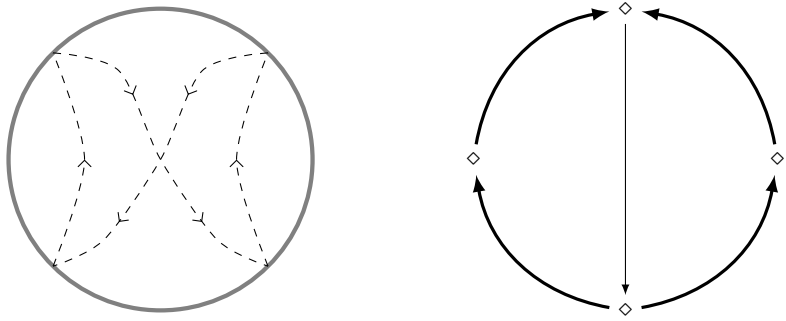

As is the case for any algebra of finite global dimension, $B_D$ is Iwanaga--Gorenstein (of Gorenstein dimension equal to the global dimension, so in this case $2$), and a $B_D$-module is Gorenstein projective if and only if it is projective. However, one can observe that the radical of the projective $B_D$-module at the lowest vertex (the head of the unfrozen arrow) is not projective, but it is free and finitely-generated over $Z$, i.e.\ Cohen--Macaulay, since it is a submodule of a projective $B_D$-module, which is Cohen--Macaulay by Proposition~\ref{p:thin}. Hence in this case the category of Gorenstein projective modules for $B_D$ is a proper subcategory of the category of Cohen--Macaulay modules.
\end{eg}


\section{\texorpdfstring{Loops and $2$-cycles}{Loops and 2-cycles}}
\label{s:no-loops}

Let $B=B_D$ be the boundary algebra of a connected Postnikov diagram $D$, and let $T_D=eA_D\in\GP(B)$ be the cluster-tilting object from Theorem~\ref{t:categorification}. In this section we show that the quiver of $\Endalg{B}{T}$ has no loops or $2$-cycles for any cluster-tilting object $T\in\GP(B)$ mutation equivalent to $T_D$---so far we know this only for the relatively small number of such objects arising from the Postnikov diagrams related to $D$ by sequences of geometric exchanges. This will allow us to complete the proof of Theorem~\ref{t:mainthm} in the next section, by giving us access to results of Fu and Keller \cite{fucluster}. The statement requires our non-degeneracy assumption that $D$ has at least $3$ strands (cf.~\cite[Prop.~5.17]{presslandmutation}).

For the remainder of the section, we fix a cluster-tilting object $T_0\in\GP(B)$, and write $A_0=\Endalg{B}{T_0}$. We assume that
\begin{enumerate}
\item $A_0$ is isomorphic to the frozen Jacobian algebra of an ice quiver with potential $(Q,F,W)$, in such a way that the vertices of $F$ correspond to projection onto indecomposable projective summands of $T_0$,
\item the quiver $Q$ has no loops or $2$-cycles, and
\item $A_0/\Span{\idemp{i}}$ is finite-dimensional whenever $i\in F_0$.
\end{enumerate}
All of these assumptions hold if $T_0=eA_D$ is the initial cluster-tilting object from Theorem~\ref{t:categorification}---the absence of loops and $2$-cycles in $Q_D$ (after replacing $D$ by a reduced diagram if necessary) follows from Proposition~\ref{p:reduced-quiv}, and finite-dimensionality in (3) follows from Proposition~\ref{p:thin} via an exactly analogous argument to that in the proof of Proposition~\ref{p:smallness}. Our goal is to show that these properties are preserved by mutations, and so hold in the entire mutation class of those cluster-tilting objects coming from Postnikov diagrams.

With this in mind, we additionally fix for the rest of the section a cluster-tilting object $T$ obtained from $T_0$ by mutation at a non-projective indecomposable summand, and write $A=\Endalg{B}{T}$. By \cite[Thm.~5.14]{presslandmutation}, the algebra $A$ again satisfies (1), with the required ice quiver with potential being the mutation $\mu_k(Q,F,W)=(Q',F',W')$ at an appropriate vertex $k$. Thus $Q'$ has no loops, and no $2$-cycles incident with $k$, but a priori it may have other $2$-cycles. The algebra $A$ satisfies (3), as the following lemma shows.

\begin{lem}
\label{l:finite-dim}
Let $A_0$ be the frozen Jacobian algebra of an ice quiver with potential $(Q,F,W)$, and assume $i$ is a frozen vertex such that $A_0/\Span{\idemp{i}}$ is finite-dimensional. If $A$ is the frozen Jacobian algebra of $\mu_k(Q,F,W)$ for some mutable vertex $k$, then $A/\Span{\idemp{i}}$ is again finite-dimensional.
\end{lem}
\begin{proof}
Note that the quivers $Q$ and $\mu_kQ$ have the same set of frozen vertices, so the statement makes sense. It is straightforward to check $A_0/\Span{\idemp{i}}$ and $A/\Span{\idemp{i}}$ are, like $A_0$ and $A$, frozen Jacobian algebras of ice quivers with potential related by mutation at $k$. (The required ice quivers with potential are obtained from those of $A_0$ and $A$ by removing the vertex $i$ and all incident arrows from the quiver, and removing from the potential all terms given by cycles through $i$.) That mutations preserve finite dimensionality of frozen Jacobian algebras follows via a straightforward adaptation of the proof of \cite[Cor.~6.6]{derksenquivers1} for ordinary Jacobian algebras.
\end{proof}

Thus it remains to show that the quiver $Q'$ has no $2$-cycles. We start by showing that $A/\Span{\idemp{i}}$ has finite global dimension for any frozen vertex $i$.

\begin{prop}
\label{p:subalg-is-Z}
Let $i$ be a frozen vertex. Then $\idemp{i}A\idemp{i}\iso Z$.
\end{prop}
\begin{proof}
Since $\idemp{i}$ corresponds to projection onto an indecomposable projective summand $P_i$ of $T$, we have $\idemp{i}A\idemp{i}=\Endalg{B}{P_i}$. Since $P_i$ is an indecomposable summand of every cluster-tilting object of $\GP(B)$, including $T_D$, this endomorphism algebra is isomorphic to $Z$ by Proposition~\ref{p:thin} and Theorem~\ref{t:categorification}.
\end{proof}

In the setting of Proposition~\ref{p:subalg-is-Z}, we thus have a recollement
\begin{equation}
\label{eq:recollement}
\begin{tikzcd}[column sep=40pt]
\fgmod{A/(\idemp{i})}\arrow["\iota" description]{r}&\fgmod{A}\arrow["\idemp{i}" description]{r}\arrow[shift right=3]{l}\arrow[shift left=3]{l}&\fgmod{Z}\arrow[shift right=3,swap]{l}{\ell}\arrow[shift left=3]{l}
\end{tikzcd}
\end{equation}
of abelian categories, in which $\iota$ is the natural inclusion, $\idemp{i}\colon X\mapsto \idemp{i}X$ and $\ell=A\idemp{i}\tens_Z-$ is left adjoint to $\idemp{i}$.

\begin{lem}
\label{l:homological-embedding}
In the recollement \eqref{eq:recollement}, the functor $\iota$ is a homological embedding, meaning it induces isomorphisms on all extension groups. As a consequence, $\gldim{A/\Span{\idemp{i}}}\leq\gldim{A}\leq 3$.
\end{lem}
\begin{proof}
We use a result of Psaroudakis \cite[Prop.~4.15]{psaroudakishomological}, implying that it is enough to check that the image, denoted by $\mathsf{F}(P)$ in \cite{psaroudakishomological}, of the counit
\[\varepsilon\colon\ell(\idemp{i}P)=A\idemp{i}\tens_Z\idemp{i}P\to P,\]
is projective for all projective $P\in\fgmod{A}$. We may write $P=\Hom_{B}(T,T')$ for $T'\in\add{T}\subset\GP(B)\subset\CM(B)$, so that $\idemp{i}P=\Hom_B(P_i,T')$ is a free $Z$-module. Thus $A\idemp{i}\tens_Z\idemp{i}P\iso(A\idemp{i})^n$, where $n$ is the rank of the $Z$-module $\idemp{i}P$, is projective. We claim that the counit $\varepsilon$ is injective, so that its image is isomorphic to the projective module $(A\idemp{i})^n$.

To see this, observe that $A\idemp{i}\tens_Z\idemp{i}P\iso(A\idemp{i})^n$ and $P$ are Cohen--Macaulay $A$-modules---that is, they are free and finitely generated over $Z$---and so this property is inherited by the kernel $K$ of $\varepsilon$. On the other hand, $\idemp{i}\varepsilon\colon \idemp{i}A\idemp{i}\tens_Z \idemp{i}P\to \idemp{i}P$ is an isomorphism by Proposition~\ref{p:subalg-is-Z}, and so $\idemp{i}K=0$. Thus $K$ is a finitely-generated $A/\Span{\idemp{i}}$-module, hence finite-dimensional by Lemma~\ref{l:finite-dim}. Since the only finite-dimensional Cohen--Macaulay $A$-module is $0$, we have the result.

Having established that $\iota$ is a homological embedding, it is then immediate that $\gldim{A/\Span{\idemp{i}}}\leq\gldim{A}$. That $\gldim{A}\leq3$ is a consequence of \cite[Prop.~3.7]{presslandinternally}, since $A$ is the endomorphism algebra of the ($2$-)cluster-tilting object $T\in\GP(B)$, and $\GP(B)\subseteq\fgmod(B)$ is the category of third syzygy modules since $B$ has Gorenstein dimension at most $3$ by Theorem~\ref{t:categorification}(1).
\end{proof}

\begin{prop}
\label{p:no-2-cycles-gen}
Let $J=\frjac{Q}{F}{W}$ be the frozen Jacobian algebra of a reduced ice quiver with potential $(Q,F,W)$. If $J$ is finite-dimensional and $\gldim{J}<\infty$, then $Q$ has no loops or $2$-cycles.
\end{prop}
\begin{proof}
Since $(Q,F,W)$ is reduced, $Q$ is the Gabriel quiver of $J$ \cite[Rem.~3.4]{presslandmutation}. That $Q$ has no loops then follows from the no-loops theorem \cite{lenzingnilpotente,igusanoloops}, without using that $J$ is a frozen Jacobian algebra. However, this latter assumption means that for each $i\in Q_0$ there is an exact sequence
\[\bigdsum_{b\in\mHead{i}}J\idemp{\tail{b}}\to\bigdsum_{a\in\Tail{i}}J\idemp{\head{a}}\to J\idemp{i}\to S_i\to 0,\]
for $S_i=J\idemp{i}/\rad{J}\idemp{i}$ the simple $J$-module supported at vertex $i$, and this sequence is the start of a projective resolution of $S_i$ \cite[Prop.~3.3(b)]{buanmutation}. Since $Q$ has no loops, $\tail{b}\ne i$ for all $b\in\mHead{i}$, and so
\[\Hom_J\Bigg(\bigdsum_{b\in\mHead{i}}J\idemp{\tail{b}},S_i\Bigg)=0.\]
It follows that $\Ext^2_J(S_i,S_i)$, which is a subquotient of this space, is also zero. That $Q$ has no $2$-cycles then follows from \cite[Prop.~3.11]{geissrigid}.
\end{proof}

We now return to $A=\End_B(T)^\op$, for $T$ obtained from $T_0$ by a mutation at a non-projective indecomposable summand. Recall that $A\iso\frjac{Q'}{F'}{W'}$, where $(Q',F',W')$ is obtained by mutation from $(Q,F,W)$ and is thus reduced by definition.

\begin{prop}
\label{p:no-2-cycles}
The quiver $Q'$ has no $2$-cycles.
\end{prop}
\begin{proof}
Let $i$ be a frozen vertex, and $\idemp{i}\in A$ the corresponding idempotent. As in the proof of Lemma~\ref{l:finite-dim}, the algebra $A/\Span{\idemp{i}}$ is again a frozen Jacobian algebra. It is finite-dimensional by the same lemma, and has finite global dimension by Lemma~\ref{l:homological-embedding}. Thus by Proposition~\ref{p:no-2-cycles-gen}, its quiver has no $2$-cycles. Since this quiver is obtained from $Q'$ by deleting the vertex $i$ and all incident arrows, there are no $2$-cycles in $Q'$ incident with $i$. Since $D$ has at least $3$ boundary regions, $Q'$ has at least three frozen vertices and so by repeating this process we rule out $2$-cycles in $Q'$ altogether.
\end{proof}

\begin{rem}
Our standing assumption that Postnikov diagrams have at least $3$ strands is only used here to rule out loops and $2$-cycles incident only with frozen vertices. It is not needed to rule out loops and $2$-cycles incident with mutable vertices---indeed, if there are fewer than $3$ strands, there are no mutable vertices anyway.
\end{rem}

We sum up our arguments in the following theorem.

\begin{thm}
\label{t:no-loops}
Let $D$ be a reduced and connected Postnikov diagram. Let $B$ be the boundary algebra of $D$, and let $T_D\in\GP(B)$ be the initial cluster-tilting object from Theorem~\ref{t:categorification}. If $T\in\GP(B)$ is a cluster-tilting object obtained from $T_D$ by a sequence of mutations, then $\Endalg{B}{T}$ is isomorphic to the frozen Jacobian algebra of an ice quiver with potential $(Q,F,W)$ such that $Q$ has no loops and no $2$-cycles.
\end{thm}
\begin{proof}
The results of this section show that the desired properties of $T$ are preserved by mutations. Since they hold in the case $T=T_D$ by Theorem~\ref{t:categorification}---reducedness of $D$ implying that $Q_D$ has no loops or $2$-cycles---we obtain the result by induction.
\end{proof}

\section{Cluster structures and cluster characters}
\label{s:cluster-character}

In this section we recall definitions and results of Buan--Iyama--Reiten--Scott \cite{buancluster} and Fu--Keller \cite{fucluster} on cluster structures and cluster characters. For our applications here, we need to weaken some definitions slightly, since the conclusions of Theorem~\ref{t:no-loops} hold only for one mutation class of cluster-tilting objects in the category $\GP(B_D)$ for a suitable Postnikov diagram $D$, whereas the original definition of a cluster structure requires these properties for all cluster-tilting objects. However, we will see in this section that (unsurprisingly) Fu--Keller's cluster character \cite{fucluster} still provides a bijection between the cluster-tilting objects in this mutation class and the clusters of $\clustalg{D}$, and between their indecomposable summands and the cluster variables of $\clustalg{D}$.

All categories in this section are assumed to be $\KK$-linear for a field $\KK$ which is of characteristic zero (as above) and now additionally algebraically closed. Note that the category $\GP(B_D)$ constructed above from a Postnikov diagram has this property, for $\KK$ the field over which the dimer algebra is defined.

\begin{defn}
Let $\frobcat$ be a Krull--Schmidt category and let $\ctcat\subseteq\frobcat$ be a full and additively closed subcategory. The quiver $Q(\ctcat)$ has vertices given by the set $\indec{\ctcat}$ of isomorphism classes of indecomposable objects in $\ctcat$, and $\dim(\catrad_{\ctcat}(T,U)/\catrad^2_{\ctcat}(T,U))$ arrows from $U$ to $T$ for each $T,U\in\indec{\ctcat}$. Here $\catrad_{\ctcat}(T,U)$ consists of non-isomorphisms $T\to U$, and $\catrad_{\ctcat}^2(T,U)$ is the subspace consisting of those expressible as a composition $T\map{f} V\map{g} U$ where $f$ is not a split monomorphism, $g$ is not a split epimorphism, and $V\in\ctcat$. We treat $Q(\ctcat)$ as an ice quiver whose frozen vertices are those corresponding to indecomposable projective objects of $\frobcat$ which lie in $\ctcat$.

If $T\in\frobcat$ is an object, we write $Q(T):=Q(\add{T})$, noting that this is the ordinary quiver of the semi-perfect algebra $\End_{\frobcat}(T)^\op$.
\end{defn}

\begin{defn}[cf.~{\cite[\S II.1]{buancluster}}, {\cite[Def.~2.4]{fucluster}}]
\label{d:cluster-structure}
Let $\frobcat$ be a Krull--Schmidt stably $2$-Calabi--Yau Frobenius category, and let $\ctcat_0\subset\frobcat$ be a cluster-tilting subcategory. We say that $(\frobcat,\ctcat_0)$ has a \emph{cluster structure} if
\begin{enumerate}
\item for each cluster-tilting subcategory $\ctcat$ of $\frobcat$ and each non-projective indecomposable $X\in\ctcat$, there is a non-projective indecomposable object $X^*$, unique up to isomorphism, such that the full and additively closed subcategory $\ctcat'=\mu_X\ctcat$, called the \emph{mutation} of $\ctcat$ at $X$ and having indecomposable objects \[\indec{\ctcat'}=\indec{\ctcat}\setminus\{X\}\cup\{X^*\},\]
 is cluster-tilting,
\item in the situation of (1), there are short exact sequences
\[\begin{tikzcd}[row sep=5pt]
0\arrow{r}&X^*\arrow{r}&U^+\arrow{r}&X\arrow{r}&0\\
0\arrow{r}&X\arrow{r}&U^-\arrow{r}&X^*\arrow{r}&0
\end{tikzcd}\]
in which the non-zero maps are minimal $(\ctcat\cap\ctcat')$-approximations (on the left or the right as appropriate),
\item for any cluster-tilting subcategory $\ctcat$ in the mutation class of $\ctcat_0$ (i.e.\ related to $\ctcat_0$ by a sequence of mutations as in (1)), the quiver $Q(\ctcat)$ has no loops or $2$-cycles, and
\item for $\ctcat$ as in (3) and $X\in\ctcat$ non-projective, the quiver $Q(\mu_X\ctcat)$ agrees, up to arrows between frozen vertices, with the Fomin--Zelevinsky mutation $\mu_XQ(\ctcat)$.
\end{enumerate}
If $T_0$ is a cluster-tilting object, then the additive closure $\add{T_0}$ is a cluster-tilting subcategory, and we say that $(\frobcat,T_0)$ has a cluster structure if $(\frobcat,\add{T_0})$ does. Cluster-tilting subcategories $\ctcat$ in the mutation class of $\ctcat_0$, and their objects, are said to be \emph{reachable} from $\ctcat_0$.

Under our assumptions on $\frobcat$, conditions (1) and (2) of this definition are automatic by \cite[Thm.~II.1.10(a)]{buancluster} (see also \cite[Thms.~5.1, 5.3]{iyamamutation}), but we include them in the definition anyway to aid comparison with \cite{buancluster}.
\end{defn}

\begin{rem}
In making Definition~\ref{d:cluster-structure}, we follow \cite{fucluster}, weakening conditions (3) and (4) to restrict to a single mutation class of cluster-tilting subcategories. Thus $\frobcat$ has a cluster structure, in the sense of \cite[Def.~2.4]{fucluster}, if and only if $\frobcat$ has a cluster-tilting subcategory and $(\frobcat,\ctcat)$ has a cluster structure for all such subcategories $\ctcat$. In principle, one could weaken conditions (1) and (2) in a similar way, but this seems to be less useful since these conditions are known to hold as stated in wide generality. The original definition from \cite{buancluster} applies to a larger class of categories, but we will not need this generality here.
\end{rem}

\begin{lem}
\label{l:mutation}
Assume that $(\frobcat,\ctcat_0)$ has a cluster structure. If $\ctcat$ is reachable, and $X\in\ctcat$ is indecomposable and non-projective, then the object $X^*$ as in Definition~\ref{d:cluster-structure}(1) satisfies
\[\dim\Ext^1_\frobcat(X,X^*)=\dim\Ext^1_\frobcat(X^*,X)=1\]
and the exact sequences in Definition~\ref{d:cluster-structure}(2) are not split. Moreover, if we write $B(X,Y)$ for the number of arrows $X\to Y$ in $Q(\ctcat)$, the middle terms of the exact sequences in Definition~\ref{d:cluster-structure}(2) are
\[U^+\iso\bigdsum_{Y\in\indec{\ctcat}}N^{B(X,Y)},\qquad U^-\iso\bigdsum_{Y\in\indec{\ctcat}}N^{B(Y,X)}\]
\end{lem}
\begin{proof}
This is proved in exactly the same way as \cite[Lem.~2.2]{fucluster} (in particular, using the assumption that $\KK$ is algebraically closed), noting that $Q(\ctcat)$ has no loops by Definition~\ref{d:cluster-structure}(3) and the assumption that $\ctcat$ is reachable.
\end{proof}

\begin{defn}[cf.~{\cite[Def.~3.1]{fucluster}}]
\label{d:clustchar}
Let $\frobcat$ be a Krull--Schmidt stably $2$-Calabi--Yau Frobenius category, and let $R$ be a commutative ring. A \emph{cluster character} is a map $\Phi\colon\Ob(\frobcat)\to R$ such that
\begin{enumerate}
\item if $X\iso X'$ then $\Phi(X)=\Phi(X')$,
\item we have $\Phi(X\oplus Y)=\Phi(X)\Phi(Y)$ for all $X,Y\in\Ob(\frobcat)$, and
\item if $X,Y\in\Ob(\frobcat)$ satisfy $\dim\Ext^1_\frobcat(X,Y)=1$ (and hence $\dim\Ext^1_\frobcat(Y,X)=1$ by the Calabi--Yau property) and
\[\begin{tikzcd}[row sep=5pt]
0\arrow{r}&X\arrow{r}&E^+\arrow{r}&Y\arrow{r}&0\\
0\arrow{r}&Y\arrow{r}&E^-\arrow{r}&X\arrow{r}&0
\end{tikzcd}\]
are non-split sequences, then
\[\Phi(X)\Phi(Y)=\Phi(E^+)+\Phi(E^-).\]
\end{enumerate}
Below, we will often employ a standard abuse of notation and write $\Phi\colon\frobcat\to R$.
\end{defn}

The next definition is due to Fu and Keller \cite{fucluster}, following Palu \cite{palucluster} and Caldero--Chapoton \cite{calderocluster} for triangulated categories, and describes a standard cluster character on a suitable Frobenius category with a cluster-tilting object.

\begin{defn}[cf.~{\cite[Rem.~3.5]{fucluster}}]
	\label{d:FK-clustchar}
	Let $\frobcat$ be a Krull--Schmidt stably $2$-Calabi--Yau Frobenius category. Assume $T\in\frobcat$ is a cluster-tilting object such that $A=\End_\frobcat(T)^\op$ is a Noetherian algebra of finite global dimension. The stable endomorphism algebra $\stab{A}=\stabEnd_\frobcat(T)^\op$ is finite-dimensional since $\frobcat$ is stably $2$-Calabi--Yau, and we view $\fgmod{\stab{A}}$ as a subcategory of $\fgmod{A}$ via the quotient map. Writing \begin{align*}
		F&=\Hom_\frobcat(T,-)\colon\frobcat\to\fgmod A,\\
		G&=\Ext^1_\frobcat(T,-)\colon\frobcat\to\fgmod\stab{A},
	\end{align*}
	we define, for each $X\in\frobcat$,
	\begin{equation}
		\label{eq:FKCC}
		\FKcc{T}(X)=x^{[FX]}\sum_d\chi(\QGr{d}{GX})x^{-[d]}\in\CC[\K_0(\per{A})].
	\end{equation}
	Here the sum is taken over possible dimension vectors $d$ of $\stab{A}$-modules, $\QGr{e}{GX}$ denotes the Grassmannian of submodules of $GX$ of dimension vector $d$, and $\chi$ denotes the Euler characteristic. Moreover, $[M]$ denotes the class of an $A$-module $M$ in the Grothendieck group $\K_0(\per{A})=\K_0(\fgmod{A})$, and $[d]=[N]\in\K_0(\per{A})$ for any module $N$ of dimension vector $d$, this class being independent of the choice of $N$ as in the proof of \cite[Prop.~3.2]{fucluster}.
\end{defn}

\begin{rem}
\label{r:CC-rewrite}
We may assume $T$ is basic, and choose a decomposition $T=\bigoplus_{i=1}^nT_i$ of $T$ into indecomposable direct summands. This yields a basis (strictly speaking a free generating set) $[P_i]=[FT_i]$ of $\K_0(\per{A})$ consisting of the classes of indecomposable projective $A$-modules, and allows us to write $\FKcc{T}(X)$ as a Laurent polynomial in the variables $x_i=x^{[P_i]}$. Since $GT_i=0$, it follows immediately that $\FKcc{T}(T_i)=x_i$. This preference for the basis of indecomposable projectives is why we usually write $\K_0(\per{A})$ instead of $\K_0(\fgmod{A})$ for the lattice of exponents.

However, thinking instead in terms of $\K_0(\fgmod{A})$, we may use the non-degenerate Euler pairing
\[\K_0(\fgmod{A})\times\K_0(\fd{A})\to\ZZ:([M],[N])\mapsto\Eul{M}{N}:=\sum_{j\geq0}\dim\Ext^j_A(M,N)\]
between the Grothendieck groups of finitely generated and finite-dimensional $A$-modules, well-defined since $A$ has finite global dimension. This pairing exhibits the classes of simple modules $S_i=FT_i/\catrad_\frobcat(T,T_i)$ as a basis of $\K_0(\fd{A})$ dual to the basis of $\K_0(\fgmod{A})=\K_0(\per{A})$ given by the classes of indecomposable projectives $P_i=FT_i$. Thus we may write
\[\FKcc{T}(X)=\prod_{i=1}^nx_i^{\Eul{FX}{S_i}}\sum_d\chi(\QGr{d}{GX})\prod_{i=1}^nx_i^{-\Eul{d}{S_i}}.\]
Since $d$ is a dimension vector for $\stab{A}$, we have $-\Eul{d}{S_i}=\Eul{S_i}{d}$ by the relative (or internal) Calabi--Yau property of $A$ \cite{kellerclustertilted}. In this way we recover the formula as stated in \cite[Rem.~3.5]{fucluster}.
\end{rem}

\begin{thm}[cf.~{\cite[Thm.~3.3]{fucluster}}]
\label{t:FKcc}
For $\frobcat$ and $T$ as in Definition~\ref{d:FK-clustchar}, the function $\FKcc{T}\colon\frobcat\to\CC[\K_0(\per{A})]$ is a cluster character on $\frobcat$ in the sense of Definition~\ref{d:clustchar}.
\end{thm}
\begin{proof}
The argument given for \cite[Thm.~3.3]{fucluster} applies equally well here; while Fu and Keller assume that $\frobcat$ is Hom-finite, our assumption that $A=\End_\frobcat(T)^\op$ is Noetherian is sufficient for their proof to go through. (For example, it is still the case that any finitely generated $A$-module, such as a finite-dimensional $\stab{A}$-module, is finitely presented.) Indeed, our additional assumption that $\gldim{A}<\infty$ even allows some parts of the proof in \cite{fucluster} to be simplified.
\end{proof}

\begin{defn}[cf.~{\cite[Def.~5.1]{fucluster}}]
\label{d:2cy-realisation}
Let $\frobcat$ be a Krull--Schmidt stably $2$-Calabi--Yau Frobenius category with cluster-tilting object $T$, and let $\clustalg{Q}$ be the cluster algebra determined by an ice quiver $Q$. We say $(\frobcat,T)$ is a \emph{Frobenius $2$-Calabi--Yau realisation} of $\clustalg{Q}$ if $(\frobcat,T)$ has a cluster structure in the sense of Definition~\ref{d:cluster-structure} and the quiver $Q(T)$ is isomorphic to $Q$ as an ice quiver, up to arrows between frozen vertices.
\end{defn}

Note that Definition~\ref{d:2cy-realisation} is slightly weaker than \cite[Def.~5.1]{fucluster}, since we have weakened the definition of a cluster structure. For example, we do not require that $Q(T')$ has no loops or $2$-cycles if $T'$ is a non-reachable cluster-tilting object. However, we still have the following important result.

\begin{thm}[cf.~{\cite[Thm.~5.4]{fucluster}}]
\label{t:clust-bijections}
Let $\clustalg{Q}$ be a cluster algebra with Frobenius $2$-Calabi--Yau realisation $(\frobcat,T)$, and write $A=\End_\frobcat(T)^\op$. The Grothendieck group $\K_0(\per{A})$ has a basis given by the classes $[P_i]$, for $i\in Q_0$, of indecomposable projective $A$-modules, and so we may view $\clustalg{Q}$ as a subalgebra of the Laurent polynomial ring $\CC[\K_0(\per{A})]$ by identifying $x_i=x^{[P_i]}$ with the initial cluster variables. Then the cluster character $\FKcc{T}$ induces bijections
\begin{enumerate}
\item from the set of isomorphism classes of indecomposable projective objects in $\frobcat$ to the set of frozen variables of $\clustalg{Q}$,
\item from the set of isomorphism classes of reachable rigid indecomposable non-projective objects of $(\frobcat,T_0)$ to the set of non-frozen cluster variables of $\clustalg{Q}$, and
\item from the set of isomorphism classes of reachable cluster-tilting objects of $(\frobcat,T_0)$ to the set of clusters of $\clustalg{Q}$.
\end{enumerate}
\end{thm}
\begin{proof}
This is proved exactly as \cite[Thm.~5.4]{fucluster}, despite our weakening of the definition of a cluster structure. Indeed, the relevant maps are shown to be injective in \cite[Lem.~4.2, Prop.~4.3]{fucluster} without assuming that $\frobcat$ has a cluster structure in either sense. The proof of surjectivity (and that the maps are well-defined, i.e.\ that they take values in the appropriate sets) uses an inductive argument. We have already seen in Remark~\ref{r:CC-rewrite} that $\FKcc{T}$ takes the indecomposable summands of $T$ to the initial cluster variables $x_i$ of $\clustalg{Q}$ (and the indecomposable projective summands to the frozen variables, by the definition of the ice quiver structure on $Q(T)=Q$). The conclusion then follows by using compatibility between mutations of reachable cluster-tilting objects and mutations of clusters in $\clustalg{Q}$. For example, the quivers of endomorphism algebras transform correctly by parts (3) and (4) of Definition~\ref{d:cluster-structure} and, in the notation of Definition~\ref{d:cluster-structure}(2), we have
\[\FKcc{T}(X)\FKcc{T}(X^*)=\FKcc{T}(U^+)+\FKcc{T}(U^-)\]
by Lemma~\ref{l:mutation} and Theorem~\ref{t:FKcc}, corresponding to an exchange relation in $\clustalg{Q}$ by Lemma~\ref{l:mutation} again. The inductive nature of the argument means that we do not require any properties of non-reachable cluster-tilting objects, since these do not appear.
\end{proof}

We close this section by applying our results to the category $\GP(B_D)$ obtained from a suitable connected Postnikov diagram $D$. The resulting theorem justifies our description of $\GP(B_D)$, with its distinguished mutation class of cluster-tilting objects reachable from $T_D$, as a categorification of the cluster algebra $\clustalg{D}$.

\begin{thm}
\label{t:2cy-realisation}
Let $D$ be a connected Postnikov diagram, let $B_D$ be its boundary algebra (defined over an algebraically closed field $\mathbb{K}$ of characteristic zero), and let $T_D\in\GP(B_D)$ be the cluster-tilting object from Theorem~\ref{t:categorification}. Then $(\GP(B_D),T_D)$ is a Frobenius $2$-Calabi--Yau realisation of the cluster algebra $\clustalg{D}$.

In particular, the Fu--Keller cluster character $\FKcc{T_D}$ provides a bijection between rigid indecomposable objects reachable from $T_D$ and cluster variables of $\clustalg{D}$ (including a bijection between indecomposable projectives and frozen variables) and between cluster-tilting objects reachable from $T_D$ and clusters of $\clustalg{D}$.
\end{thm}
\begin{proof}
The category $\GP(B_D)$ is a stably $2$-Calabi--Yau Frobenius category, and $T_D\in\GP(B_D)$ a cluster-tilting object, by Theorem~\ref{t:categorification}. It is idempotent complete, and admits a fully faithful embedding to a Krull--Schmidt category by \cite[Prop.~3.6]{canakciperfect} (see Proposition~\ref{p:restriction} below), and so it is Krull--Schmidt. Moreover, $(\GP(B_D),T_D)$ has a cluster structure, as follows. Conditions (1) and (2) are automatic as discussed in Definition~\ref{d:cluster-structure}. Condition (3) holds by Theorem~\ref{t:no-loops} and, as a consequence, condition (4) holds by \cite[Thm.~5.15]{presslandmutation}. Since $Q(T_D)$ is isomorphic to $Q_D$ by Theorem~\ref{t:categorification} again, $(\GP(B_D),T_D)$ is a Frobenius $2$-Calabi--Yau realisation of $\clustalg{D}$, and the required bijections then follow from Theorem~\ref{t:clust-bijections}.
\end{proof}

\section{Relationship to the Grassmannian cluster category}
\label{s:Grassmannian}
Fix a connected Postnikov diagram $D$ of type $(k,n)$, and abbreviate $A=A_D$, $B=B_D$ and $T=T_D$. To be able to use the results of Section~\ref{s:cluster-character}, we take the ground field for these algebras to be algebraically closed.

Recall that $D$ determines an ice quiver $Q=Q_D$ with cluster algebra $\clustalg{}=\clustalg{D}$, a positroid $\posit=\posit_D$ with necklace $\neck=\neck_D$, and an open positroid variety $\openposvar=\openposvar(\posit_D)$ with cone $\coneopenposvar$. In this final section, we explain the relationship between $\GP(B)$ and the cluster category $\CM(C)$ for the Grassmannian $\Gra{k}{n}$ introduced by Jensen, King and Su \cite{jensencategorification}. As a result, we will see that the composition of the Fu--Keller cluster character $\FKcc{T}$ on $\GP(B)$ with Galashin--Lam's isomorphism $\clustalg{}\isoto\CC[\coneopenposvar]$ (modified as in Remark~\ref{r:positroids}) is a representation-theoretically natural map.

\begin{prop}[{\cite[Prop.~3.6]{canakciperfect}}]
\label{p:restriction}
Let $\CM(C)$ be the Grassmannian cluster category for the $(k,n)$-Grassmannian. Then there is a fully faithful functor $\rho\colon\GP(B)\to\CM(C)$, arising from an injective map $C\to B$ of $Z$-orders that becomes an isomorphism under tensor product with $Z[t^{-1}]$.
\end{prop}

Recall from Remark~\ref{r:positroids} that we can label each alternating region of $D$ (or equivalently, each vertex $j\in Q_0$) by $I_j\in\subsets{k}{n}$. Precisely, the alternating region corresponding to $j$ lies to the left of $k$ strands, and $I_j$ consists of the labels of the start-points of these strands. This leads to the specialisation map $\CC[x_j^{\pm1}:j\in  Q_0]\to\CC(\coneopenposvar)$, taking each $x_j$ to $\Plueck{I_j}|_{\coneopenposvar}$, which restricts to an isomorphism $\clustalg{D}\to\CC[\coneopenposvar]$ \cite{galashinpositroid}. We write $\Phi\colon\GP(B)\to\CC[\coneopenposvar]$ for the composition of the Fu--Keller cluster character $\FKcc{T}\colon\GP(B)\to\CC[\K_0(\per{A})]=\CC[x_j^{\pm1}:j\in  Q_0]$ with this specialisation map. Since the indecomposable summand $T_j=eAe_j$ of $T$ satisfies $\FKcc{T}(T_j)=x_j$ for each $j\in Q_0$, we have $\Phi(T_j)=\Plueck{I_j}|_{\coneopenposvar}$.

For each $I\in\subsets{k}{n}$, Jensen, King and Su explicitly describe an indecomposable rigid module $M_I\in\CM(C)$. These objects categorify the Plücker coordinates $\Plueck{I}\in\CC[\coneGra{k}{n}]$, all of which are cluster variables. Over the course of this section, we will show that our cluster character $\Phi$ is compatible with the embedding $\rho$. The next proposition is crucial for this.

\begin{prop}[\cite{jensencategorification3}]
\label{p:ext-closed}
The natural map
\[\Ext^1_{B}(X,Y)\to\Ext^1_C(\rho X,\rho Y)\]
is an isomorphism whenever $X\in\CM(B)$ and $Y\in\GP(B)$. In particular, the essential image of $\rho$ is extension-closed.
\end{prop}
\begin{proof}
The map is well-defined since $\rho$ is exact, and injective since $\rho$ is fully faithful; these conclusions hold even under the weaker assumption $Y\in\CM(B)$. For surjectivity, let $P\to X$ be a projective cover of $X$ and consider the commutative diagram
\begin{equation}
\label{eq:ext-pullback}
\begin{tikzcd}
0\arrow{r}&\rho Y\arrow{r}\arrow[equal]{d}&E'\arrow{r}\arrow{d}&\rho P\arrow{d}\arrow{r}&0\\
0\arrow{r}&\rho Y\arrow{r}&E\arrow{r}&\rho X\arrow{r}&0
\end{tikzcd}
\end{equation}
in which the lower row is an extension in $\CM(C)$ and the right-hand square is a pull-back. Since $\rho$ is fully faithful, to show that the the lower row comes from an extension of $B$-modules it suffices to show that $E$ is in the essential image of $\rho$.

Note that $\rho B$ is rigid, as follows. By \cite[Prop.~8.2]{canakciperfect} we have $\rho B\iso\bigdsum_{I\in\neck}M_I$. Since the elements of $\neck$ are weakly separated \cite[Lem.~4.5]{ohweak}, the required rigidity follows from \cite[Prop.~5.6]{jensencategorification}. As a consequence, if $Y=B$ then the upper row of \eqref{eq:ext-pullback} splits, since $P\in\add(B)$. It follows that $E$ is a factor of $\rho B\oplus\rho P$ and so it lies in the essential image of $\rho$ by \cite[\S37, Ex.~10]{curtismethods1}; here we use the additional properties of $\rho$ from Proposition~\ref{p:restriction} to see that $\rho B\subset C$ and $\rho B[t^{-1}]=C[t^{-1}]$. Thus
\begin{equation}
\label{eq:ext-iso1}
\Ext^1_B(X,B)\isoto\Ext^1_C(\rho X,\rho B).
\end{equation}

Now, for general $Y\in\GP(B)$, the upper row of \eqref{eq:ext-pullback} is an element of
\[\Ext^1_C(\rho P,\rho Y)=\Kdual\Ext^1_C(\rho Y,\rho P)=\Kdual\Ext^1_B(Y,P)=0,\]
where the first equality is the $2$-Calabi--Yau property of $\CM(C)$, the second follows from \eqref{eq:ext-iso1} since $P\in\add{B}$, and the last uses that $P\in\add(B)$ and $Y\in\GP(B)$. Thus this upper row once again splits, and we conclude as before that $E$ is in the essential image of $\rho$.
\end{proof}

\begin{prop}[{\cite[Prop.~8.2]{canakciperfect}}]
\label{p:labelling}
Let $j\in Q_0$, and let $T_j=eAe_j\in\GP(B)$ be the corresponding indecomposable summand of the initial cluster tilting object $T=eA$. Then $\rho T_j\cong M_{I_j}$.
\end{prop}

Over the course of the next several results, we will show that $\Phi(M)=\Plueck{I}|_{\coneopenposvar}$ whenever $\rho(M)\cong M_I$. We begin by recalling Jensen--King--Su's cluster character on $\CM(C)$, and relating it to our cluster character on $\GP(B)$.

\begin{prop}[{\cite[Eq.~9.4]{jensencategorification}}]
\label{p:JKS-clust-char}
There is a cluster character $\Psi\colon\CM(C)\to\CC[\coneGra{k}{n}]$ such that $\Psi(M_I)=\Plueck{I}$ for every $k$-subset $I$.
\end{prop}

\begin{prop}
\label{p:clust-char-equality}
For any reachable rigid object $X\in\GP(B)$, we have
\[\Phi(X)=\Psi(\rho X)|_{\coneopenposvar}.\]
\end{prop}
\begin{proof}
First, observe that $X\mapsto\Psi(\rho X)|_{\coneopenposvar}$ defines a cluster character $\GP(B)\to\CC[\coneopenposvar]$ since $\Psi$ is a cluster character on $\CM(C)$; this uses both Proposition~\ref{p:restriction} and, to see that the restricted map still satisfies the multiplication formula from Definition~\ref{d:clustchar}(3), Proposition~\ref{p:ext-closed}.

Thus it suffices to check the required equality on the indecomposable summands of a single reachable cluster-tilting object, which we may take to be the initial object $T$. For each $j\in Q_0$, the summand $T_j=eAe_j$ satisfies
\[\Psi(\rho T_j)|_{\coneopenposvar}=\Psi(M_{I_j})|_{\coneopenposvar}=\Plueck{I_j}|_{\coneopenposvar}=\Phi(T_j)\]
by Propositions~\ref{p:labelling} and \ref{p:JKS-clust-char}.
\end{proof}

\begin{thm}
\label{t:rank-1-plueck}
If $M\in\GP(B)$ satisfies $\rho M\cong M_I$ for some $I\in\subsets{k}{n}$, then $\Phi(M)=\Plueck{I}|_{\coneopenposvar}$.
\end{thm}
\begin{proof}
By Proposition~\ref{p:JKS-clust-char} we have $\Plueck{I}=\Psi(M_I)=\Psi(\rho M)$. Thus by Proposition~\ref{p:clust-char-equality}, it suffices to check that $M$ is reachable---it is rigid by Proposition~\ref{p:ext-closed} (see also the proof of \cite[Lem.~10.4]{canakciperfect}).

Consider a geometric exchange at an alternating region $j$ of $D$ (corresponding to a mutation at the vertex $j$ of $Q$), producing a new diagram $D'$. By considering the exchange sequences as in Definition~\ref{d:cluster-structure}(2), we may see that the mutation $T_j'$ of the direct summand $T_k$ of $T$ satisfies $\rho(T_j')=M_{I_j'}$. Since $\rho$ is fully faithful, $T_j'$ is also the unique such object in $\GP(B)$ up to isomorphism. Inductively, it follows that $M\in\GP(B)$ with $\rho(M)\cong M_I$ is reachable as long as $I$ appears as a label in a Postnikov diagram related to $D$ by a sequence of geometric exchanges. By \cite[Thm.~6.6]{ohweak}, such subsets $I$ are precisely those appearing in some maximal weakly separated collection $\clust$ with $\neck\subset\clust\subset\posit$, i.e.\ they are those $I\in\posit$ which are weakly separated from every element of $\neck$.

So assume $M\in\GP(B)$ has $\rho(M)\cong M_I$. Since $M\in\CM(B)$, it follows from \cite[Prop.~8.6]{canakciperfect} that $I\in\posit_D$. Since $M$ is even Gorenstein projective, it follows from Proposition~\ref{p:ext-closed} that
\[0=\Ext^1_{B}(M,T_j)=\Ext^1_C(M_I,M_{I_j})\]
for $T_j=eA_De_j$ with $j$ a frozen vertex, these being precisely the indecomposable summands of $B$. By \cite[Prop.~5.6]{jensencategorification}, this means precisely that $I$ and $I_j$ are weakly separated. Since $\neck$ is the set of labels $I_j$ of frozen vertices $j$, we have the result.
\end{proof}

\begin{rem}
We do not yet know whether or not $\Phi(X)=\Psi(\rho X)|_{\coneopenposvar}$ for all $X\in\GP(B)$. However, we expect that this equality does hold, even for all $X\in\CM(B)$; note that the Fu--Keller cluster character $\FKcc{T}$, and hence $\Phi$, may be defined on this larger category via the same formula \eqref{eq:FKCC}, although it is only a cluster character on the stably $2$-Calabi--Yau Frobenius category $\GP(B)$. If this were the case, the conclusion of Theorem~\ref{t:rank-1-plueck} would hold under the weaker assumption that $M\in\CM(B)$. The $k$-subsets $I$ appearing in the resulting statement would then be all of the elements of the positroid $\posit$, including those not weakly separated from the necklace, by \cite[Prop.~8.6]{canakciperfect}.
\end{rem}

\section*{Acknowledgements}
Some important realisations concerning the results of this paper were made during discussions with İlke Çanakçı and Alastair King concerning our work in \cite{canakciperfect}, and we are grateful to them for the inspiration. We also thank Chris Fraser, Bernt Tore Jensen, Andrea Pasquali, Khrystyna Serhiyenko, Melissa Sherman-Bennett and Lauren Williams for useful conversations. Additional thanks are due to Bernt Tore Jensen for informing us of the statement and proof of Proposition~\ref{p:ext-closed}. The author was supported by the EPSRC postdoctoral fellowship EP/T001771/1 while expanding and revising this paper. The author would like to thank the Isaac Newton Institute for Mathematical Sciences, Cambridge, for support and hospitality during the programme \emph{Cluster Algebras and Representation Theory}, supported by EPSRC grant no.\ EP/R014604/1, where some of these revisions were undertaken.

\defbibheading{bibliography}[\refname]{\section*{#1}}\printbibliography
\end{document}